\documentclass[11pt]{amsart}
\usepackage{amsmath,amssymb,amsthm}

\usepackage{amsmath, amssymb}
\usepackage{amsthm, amsfonts, mathrsfs}
\usepackage{fullpage}
\usepackage{amsfonts,graphicx}

\theoremstyle{plain}
\newtheorem{Theo}{Theorem}[section]
\newtheorem{lem}[Theo]{Lemma}

\newtheorem{prop}[Theo]{Proposition}

\theoremstyle{plain}
\theoremstyle{definition}

\theoremstyle{remark}
\newtheorem{Rema}[Theo]{Remark}
\newtheorem*{rema*}{Remark}

\newcommand{\NN}{\mathbb{N}}

\newcommand{\RR}{\mathbb{R}}
\newcommand{\rz}{\mathcal{R}_{ij}}

\author[T. Hmidi]{Taoufik Hmidi}
\address{IRMAR, Universit\'e de Rennes 1\\ Campus de
Beaulieu\\ 35~042 Rennes cedex\\ France}
\email{thmidi@univ-rennes1.fr}
\author[F. Rousset]{Fr\'ed\'eric Rousset}
\address{IRMAR, Universit\'e de Rennes 1\\ Campus de
Beaulieu\\ 35~042 Rennes cedex\\ France}
\email{frederic.rousset@univ-rennes1.fr}

\date{}

\begin{document}
\title[Axisymmetric Boussinesq system]
{Global well-posedness for the Euler-Boussinesq system with axisymmetric data}
\maketitle
\begin{abstract}
In this paper we prove  the global well-posedness   for the   three-dimensional Euler-Boussinesq system   with axisymmetric initial data without swirl. This system couples  the Euler equation with a transport-diffusion  equation governing the temperature. \end{abstract}
\tableofcontents
\section{Introduction}
  Boussinesq systems  are   widely used to model the dynamics of the  ocean or  the atmosphere.
    They   arise from the  
    density dependent   fluid  equations by using the so-called
    Boussinesq approximation which consists in neglecting the density dependence in all the terms
     but  the one involving the gravity.   This approximation can be justified   from compressible fluid equations
       by  a simultaneous low Mach number/Froude number limit, we refer to \cite{Feireisl} for a rigorous
        justification. In this paper we shall  assume that the fluid is inviscid 
         but heat-conducting and hence the system
         reads
     
\begin{equation}
    \label{bsintro}
\left\{ \begin{array}{ll}
\partial_{t}v+v\cdot\nabla v +\nabla p=\rho e_z,\quad (t,x)\in \RR_+\times\RR^3,\\
\partial_{t}\rho+v\cdot\nabla\rho-  \Delta \rho=0\\
\textnormal{div}\,  v=0,\\
v_{| t=0}=v_{0},   \quad \rho_{| t=0}=\rho_{0}.
\end{array} \right.
\end{equation}
Here, the velocity  $v=(v^1,v^2,v^3)$ is a three-component vector field with zero divergence,  the scalar
function $\rho$ denotes the density or the temperature and $p$ the pressure of the fluid. 
  Note that we have assumed that the heat conductivity  coefficient
     is one, one can always reduce the problem to this situation by a change of scale
      (as soon as the fluid is assumed to be heat conducting)  
      which is not important for global well-posedness issues with data of arbitrary size
       that we shall consider here.
  The  term  $\rho e_{z}$  where $e_{z}= (0, 0, 1)^t$ takes into account the influence of
   the gravity  and the stratification  on the motion of the fluid. 
Note 
 that  when the initial density $\rho_0$ is identically zero (or constant)  then the above system  reduces  to the
classical  incompressible Euler  equation:
\begin{equation}
      \label{ns}
\left\{ \begin{array}{ll}
\partial_{t}v+v\cdot\nabla v +\nabla p=0\\
\textnormal{div}\,  v=0\\
v_{| t=0}=v_{0}.
\end{array} \right.
\end{equation}
 From this observation, one cannot expect to have a better theory for the Boussinesq system
  than for the  Euler equation.
   For the Euler equation, a well-known criterion for the existence of global smooth
    solution is the Beale-Kato-Majda criterion \cite{bkm}. It  states  that
     the control of the vorticity of the fluid $\omega=  {\rm curl}\,  v$ 
      in  $L^1_{loc}(\mathbb{R}_{+}, L^\infty)$ is sufficient to get global well-posedness. 
     In space dimension two,  the vorticity $\omega$ can be identified to a scalar function
      which solves the transport equation
      $$ \partial_{t} \omega + v \cdot \nabla \omega= 0.$$
       From this transport equation, one immediately gets that
       $$ \|\omega(t) \|_{L^p} \leq \|\omega_{0}\|_{L^p}$$
         for every $p \geq 1$ and hence the global well-posedness follows from the 
       Beale-Kato-Majda criterion.
          
  In a similar way, the global well-posedness  for two-dimensional Boussinesq systems
   which has recently drawn a lot of attention seems   to be in a satisfactory state. More precisely 
 global  well-posedness   has been shown  in  various  function spaces  and for
different viscosities, we refer  for example to \cite{ah,Brenier,
ha,dp1,dp,hk1,hk,HKR1,HKR2, HZ}. In particular, for the model \eqref{bsintro} in 2D, 
   the main idea is that  by  studying carefully the coupling between the two equations 
    and by using the smoothing effect of the second equation, it is still possible to 
      get an a priori estimate  in $L^\infty$ (or in $B_{\infty, 1}^0$ ) for $\omega$ and hence the global well-posedness.
      
 In the three-dimensional case, very few is known:    even for  the Euler equation,   the vorticity $\omega$ solves
  the equation
  \begin{equation}
  \label{vor}
  \partial_{t} \omega + v \cdot \nabla \omega= \omega \cdot \nabla v 
  \end{equation}
 and   the way to control the vortex stretching term $\omega \cdot \nabla v$ in the right-hand side
   is a widely open problem. Nevertheless, 
    a classical situation   where  one can get global existence  is the
     case  that $v$ is axisymmetric without swirl  \cite{Ukhovskii},  \cite{LA}.
    Our aim here is to  study  how this  classical global existence  result for axisymmetric data
     for the Euler equation  can be extended  to  the Boussinesq system
      \eqref{bsintro}.

           Before stating  our main result, let us recall the main ingredient in the  global existence proof for the Euler equation with axisymmetric data.
          The assumption that  the vector field $v$ is axisymmetric without swirl means that    it has
the form:
\begin{equation}
\label{vform}
v(t,x) = v^r(t,r, z)e_r + v^z (t,r, z)e_z,\quad x=(x_{1},x_{2},z),\quad r=({x_{1}^2+x_{2}^2})^{\frac12},
\end{equation}
where  $\big(e_r, e_{\theta} , e_z\big)$ is the local basis of $ \mathbb{R}^3$  corresponding to cylindrical 
 coordinates. Note that we assume that  the velocity is  invariant by rotation around the  vertical axis
(axisymmetric flow) and that the angular component $v^{\theta}$ of $v$ is identically zero
 (without swirl).   For these   flows, 
  the vorticity is under the form 
$$
\omega=(\partial_3v^r-\partial_rv^z) e_{\theta}:=\omega_\theta e_\theta
$$
and  the vortex  stretching term reads
$$ \omega \cdot \nabla v= {v^r \over r }\omega.$$ 
 In particular   $\omega_\theta$ satisfies  the equation
\begin{equation}
 \label{tourbillon}
\partial_t \omega_\theta +v\cdot\nabla\omega_\theta
 =\frac{v^r}{r}\omega_\theta.
\end{equation}
 The crucial fact is then that  the quantity $\zeta:=\frac{\omega_\theta}{r}$  solves  the 
  transport equation
\begin{equation*}
\partial_t \zeta+v\cdot\nabla \zeta=0
\end{equation*}
from which we  get 
 that for every  $p\in[1,\infty]$
\begin{equation}
\label{zetaintro}
\|\zeta(t)\|_{L^p}\le\| \zeta_0\|_{L^p}.
\end{equation}

 It was shown by M. Ukhoviskii and V. Yudovich \cite{Ukhovskii}  and independently by O. A. Ladyzhenskaya \cite{LA}  that these  new a priori estimates  are  strong enough to prevent the formation of singularities in
finite time for axisymmetric flows without swirl. More precisely global existence and uniqueness was established  for axisymmetric initial data with finite energy and satisfying in addition  $\omega_{0}\in L^2\cap L^\infty$ and $\frac{\omega_{0}}{r}\in L^2\cap L^\infty.  $ In terms of Sobolev regularity these assumptions are satisfied if the velocity   $v_{0}$ belongs to $ H^s$ with $s>{7\over 2}$.   This  condition was improved   more recently.
 In \cite{Taira},     it was proven   that global  well-posedness 
  still holds   if  $v_{0}$ is in $H^s$  with  $s>\frac52$
 (note that this is the natural regularity  requirement  for the initial velocity in the Sobolev scale
  in view of the standard  local existence result) and in the recent work  \cite{rd}, Danchin
   has obtained  global existence and uniqueness for initial data such that   $\omega_0\in L^{3,1}\cap L^\infty$ and $\zeta_0\in L^{3,1}  $(here,  $L^{3,1}$  denotes the Lorentz space, 
   the definition of the Lorentz spaces $L^{p,q}$ as interpolation spaces is recalled below).
Their proof is based on the observation that  one can deduce from the Biot-Savart law, 
   the  pointwise estimate
\begin{equation}
\label{pointwise}\Big| {v^r \over r } \Big|  \lesssim  {1 \over | \cdot |^2} * | \zeta |.
\end{equation}
 By convolution laws in Lorentz spaces (again recalled below) and \eqref{zetaintro}, this yields the estimate
  $$  \Big\| {v^r  \over r }(t) \Big\|_{L^\infty}  \lesssim  \|\zeta(t) \|_{L^{3, 1}} \lesssim \| \zeta_{0} \|_{L^{3, 1}}.$$
  Since  one gets from a crude estimate on  \eqref{tourbillon} that
  $$ \|\omega_{\theta}(t) \|_{L^\infty} \leq  \|\omega_{0}(t) \|_{L^\infty} e^{\int_{0}^t  \|v^r / r \|_{L^\infty}},$$
  the global  well-posedness  in   $ H^s$ $s>5/2$ (the assumption
    $\zeta_{0}\in L^{3, 1}$ is automatically satisfied)  then follows  from  the Beale-Kato-Majda criterion.
  It is actually possible, as shown in  \cite{A-H-K},    to get  global well-posedness in the  critical Besov regularity, that is, $v_0\in B_{p,1}^{\frac3p+1}, \,\forall p\in[1,\infty]$, in the sense that  it is possible to propagate globally
   the critical Besov regularity if $\zeta_{0}\in L^{3, 1}$.

 Our aim here is to  extend these global well-posedness results  to the Boussinesq system  \eqref{bsintro}.
    Our main result reads: 
     \begin{Theo}\label{thm1}
 Consider the Boussinesq system \eqref{bsintro}.
 Let $s>\frac52$, $v_0\in H^s$ be an axisymmetric divergence free  vector field without swirl  
 and let $\rho_0$ be an axisymmetric function belonging to $ H^{s-2}\cap L^m$ with $ m>6$ and
  such that  $r^2\rho_0\in L^2$.
 Then there is  a unique global  solution $(v,\rho)$ such that
 $$
(v,\rho)\in\mathcal{C}(\RR_+;H^s)\times \Big(\mathcal{C}(\RR_+;H^{s-2}\cap L^m)\cap L^1_{\textnormal{loc}}(\RR_+; W^{1,\infty})\Big)$$
 $$
\hbox{and}\quad  r^2\rho\in \mathcal{C}(\RR_+;L^{2}).
 $$
  \end{Theo}

Let us give a few comments about our  result.
\begin{Rema}
By axisymmetric scalar function we mean again  a function that depends only on the variables $(r,z)$ but not on the angle $\theta$ in cylindrical coordinates.  One  can easily check  that for smooth local solutions, if 
 $(\rho_{0}, v_{0})$ is  axisymmetric  (and $v_{0}$ without swirl), this property is preserved
  by the evolution. 
\end{Rema}


\begin{Rema}
The assumption on the moment of $\rho$ is probably technical.
  The control of the moments of $\rho$ are needed in our proof in  some commutator estimates
  (see \eqref{comintro} for example).
\end{Rema}

Note that in view of the proof for the Euler equation,  the crucial part is to get an  a priori estimate for 
 $\zeta$ in $L^{3, 1}$.
The equation for $\zeta= \omega_{\theta}/r$
   becomes     
\begin{equation}
\label{mBouss}\partial_t \zeta+v\cdot\nabla \zeta= - {\partial_{r} \rho \over r }
\end{equation}
and  consequently,  the main  difficulty is to find  some strong a priori estimates on $\rho$ to control the  term in the right-hand side
 of  \eqref{mBouss}. 
The rough idea is that  on the axis $r=0$ the singularity ${1\over r}$
 scales  as a derivative and hence that  the forcing term $\partial_{r} \rho/r$ can be thought as a Laplacian
  of $\rho$ and thus one may try to use   smoothing effects to control it.  We observe that  
   if we  neglect for the moment the advection term $v\cdot\nabla\rho$ in the equation of the density then using the maximal smoothing effects of the heat semigroup  we can gain two derivatives by integrating in time which  is
    exactly what we need. From this point of view we see that our model is in some sense  critical for the  global well-posedness analysis.   The main difficulty  if one wants to use this argument is to  deal with  the advection term. 
     Indeed, 
      the only control on $v$ that we have at our disposal    is a $L^\infty_{loc}L^2$  estimate 
      (which comes from the basic energy estimate)
       and this is not sufficient  to obtain an estimate
           for $D^2\rho$ in $L^1_{loc}(L^p)$ by  considering the convection term as a source term
            and by using the maximal smoothing effect of the heat equation. Even more refined 
            maximal regularity  estimates on convection-diffusion equations (\cite{danchinpoche},
             \cite{hmidipoche} for example)  do not seem
             to provide 
             useful information  when the control of the velocity field  is so poor.
Consequently, our strategy for the proof   will be to use more carefully the structure of the coupling
 between the two equations of \eqref{bsintro} in order to find
  suitable a priori estimates for  $(\zeta,\, \rho)$. 
 Since  the  coupling 
 between the two equations does not make the original Boussinesq system well suited for a priori estimates, 
our  main idea is to use an approach  that  was successfully used  for  the study of   two-dimensional systems
     with a critical dissipation, see  \cite{HKR1,HKR2}
      and the Navier-Stokes-Boussinesq system with axisymmetric data  \cite{HR1}. It consists in diagonalizing
      the linear part of the system  satisfied  by $\zeta$
       and  $\rho$.
      We introduce a   new unknown $\Gamma$ which  here formally reads 
     $$  \Gamma= \zeta + {\partial_{r} \over r } \Delta^{-1}
      \rho$$
      and we  study the system satisfied by $(\Gamma, \rho)$ which is given by:
      $$ \partial_{t} \Gamma + v \cdot \nabla \Gamma= - \big[ {\partial_{r} \over r } \Delta^{-1}, v \cdot \nabla \big]
       \rho, \quad 
       \partial_{t} \rho + v \cdot \nabla \rho = \Delta \rho$$
       where $ \big[ {\partial_{r} \over r } \Delta^{-1}, v \cdot \nabla \big] 
       $ is the commutator defined by 
       $$   \big[ {\partial_{r} \over r } \Delta^{-1}, v \cdot \nabla \big] \rho=  {\partial_{r} \over r } \Delta^{-1}
        \big( v \cdot \nabla \rho \big) - v \cdot \nabla \big(  {\partial_{r} \over r } \Delta^{-1}
\rho \big).$$
Note that if we forget the commutator for a while,  we immediately get an a priori $L^p$ estimate
 for $\Gamma$  for every $p$  from which we can hope to get an $L^p$ estimate for
  $\zeta$, if the operator ${\partial_{r} \over r } \Delta^{-1}$ behaves well.   
  
 To make this argument rigorous, we  need first to study  the action of the  operator    ${\partial_{r} \over r } \Delta^{-1}$ over axisymmetric functions. This is done in Proposition \eqref{prop1} where we prove that  this operator takes the form  $$  {\partial_{r} \over r } \Delta^{-1}= \sum_{i, j} a_{ij}(x) \mathcal{R}_{ij}$$
   where $\mathcal{R}_{ij}=\partial_i\partial_j\Delta^{-1}$ are Riesz operators and the functions $a_{ij}$ 
    are bounded.
   This yields that  $ {\partial_{r} \over r } \Delta^{-1}$ acts continuously on $L^{3, 1}$
    and hence that 
    $$
   \big\|{\partial_{r} \over r } \Delta^{-1}\rho(t)\big\|_{L^{3,1}}\lesssim\|\rho(t)\|_{L^{3,1}}\lesssim\|\rho_0\|_{L^{3,1}}.
    $$ 
    It follows that the control  of  $\Gamma$ is equivalent to the control of   $\zeta$ in $L^{3, 1}$.
  Now it remains  to estimate in a suitable way  the commutator  term $ \big[ {\partial_{r} \over r } \Delta^{-1}, v \cdot \nabla \big]\rho 
       $ which is the main technical part.  It seems that there is no hope to bound the commutator without using unknown quantities because there is no other known a priori estimates of the velocity except that given by  energy estimate which is not strong enough. 
   We shall prove  (Theorem \ref{prop2}) that
   \begin{equation}
   \label{comintro}
\Big\|\big[({\partial_r}/{r})\Delta^{-1},v\cdot\nabla\big]\rho\Big\|_{L^{3,1}}\lesssim\|\omega_{\theta}/r\|_{L^{3,1}}\big(\|\rho x_h\|_{B_{\infty,1}^0\cap L^2}+\|\rho\|_{B_{2,1}^{\frac12}}\big).
\end{equation}
This estimate is  the heart of our argument, its proof  combines the use 
  of paradifferential calculus and   some harmonic analysis results  and  also requires a careful use 
   of the property that  velocity $v$ is axisymmetric without swirl in the Biot-Savart law.
   
  The main reason for  which  we need  some moments of $\rho$ in the right-hand side of \eqref{comintro}
   is that  we want an estimate of the commutator invoving $\omega_{\theta}/r$   and   not   $\omega$.
    
  In the right hand side of \eqref{comintro},      $\|\rho\|_{B_{2,1}^{\frac12}}$
   and $\|\rho x_{h} \|_{L^2}$ can be controlled in terms of the initial data only by using
   the smoothing effect of the convection-diffusion equation for $\rho$ and standard energy estimates. Consequently,
    from this commutator estimate, we obtain that
   \begin{equation}
   \label{Gammaintro} \|\zeta(t) \|_{ L^{3, 1} } \leq C(t) e^{ C \| \rho x_{h}\|_{L^1_{t} B_{\infty, 1}^0} }
   \end{equation}
    and
    the next difficult step  is to control   $\|  \rho x_{h} \|_{L^1_tB_{\infty,1}^0}$.
      This is done  in two steps. 
       The first step  is to get a global $L^\infty$ estimate of $\rho x_{h}$  in terms
       of the initial data only and then in a second step, we shall prove  a logarithmic estimate
      for    the $B_{\infty, 1}^0$ norm of $x_{h}\rho$  in terms of  the $L^{3, 1}$ norm of $\zeta$.
      
     For the first step,   let us observe that $f= \rho  x_{h}$ solves the equation
      \begin{equation}
      \label{momentintro} \partial_{t}f + v \cdot \nabla f - \Delta f= v^h\rho - 2 \nabla_{h} \rho.
      \end{equation}
       Note that for the moment, we only have at our disposal the standard energy estimate
        for $v$ (thus we control $\| v \|_{L^\infty_{t} L^2}$ only), consequently to obtain
         an  $L^\infty $ estimate for $f$ we need to use  an $L^2 \rightarrow L^\infty$
          estimate for the convection-diffusion equation since the source term in the right-hand side
           can be estimated only in $L^2$. Note that  the convection term
           cannot  be neglected (again because of   the  weak control
            on $v$ that we have  at this stage) and hence this estimate cannot be obtained
             from heat kernel estimates.  We  shall obtain this estimate  by using the Nash-Moser-De Giorgi
            iterations
            \cite{Giorgi}, \cite{Nash}, \cite{Moser}. Indeed, the main interest of this approach is that since it is based on 
             energy  type estimates, the  convection term does not contribute. A general  result is recalled in the Appendix \ref{appendix}.
             For technical reasons, some higher order moment estimates  which are easier to obtain are also needed, they are stated in Proposition \ref{prop07}.
              
           Once the estimate of $\| \rho x_{h}\|_{L^\infty}$ is known in terms of the initial data, 
           one can establish   logarithmic Besov space estimates
                        for the convection-diffusion equation \eqref{momentintro} by using a special
                         time dependent  frequency  cut-off of $x_h\rho$ where we combine  the $L^\infty$ estimate with some smoothing effects for $x_h\rho$. This yields (see \eqref{eqs444})
             \begin{equation}\label{titt1}
            \|\rho x_{h}\|_{L^1_{t} B_{\infty, 1}^0 } \leq C_{0}(t) \Big( 1+
              \int_{0}^t   h(\tau) \log\big( 2 + \|\zeta \|_{L^\infty_{\tau}L^{3, 1}}\big)\, d\tau\Big)
              \end{equation}
              where $C_0(t)$ is a given continuous function and $h$ is some $L^1_{loc}(\RR_+)$ function. We point out that the use of the moment of order two $|x_h|^2\rho$  is due to the  treatment of the commutator $[\Delta_q, v\cdot\nabla](x_h\rho)$  which appears when we deal with the smoothing effects.  
              
           The combination of  the estimates \eqref{titt1} and \eqref{Gammaintro} with   Gronwall inequality allows to \mbox{control
             $\| \zeta(t) \|_{L^{3, 1}}$} globally in time.
             
             The final step is  to  deduce,  as for the incompressible Euler equation, from
              the control of    $\| \zeta \|_{L^\infty_{t}L^{3, 1}}$ 
              an estimate  of $\| \omega \|_{L^\infty_{t }L^\infty}$  and of $\|\nabla v \|_{L^\infty_{t}L^\infty}.$
              This is the aim of Proposition  \ref{vorticity} and \ref{Soblev-est}. Estimates
               in Sobolev spaces then follow in a rather classical way.
               
             Once a priori estimates for sufficiently smooth functions are known,
              the result of Theorem \ref{thm1} follows from an approximation argument.

     The paper is organized as follows. 
   In section \ref{prelim1} we  fix the notations,  give the definitions of the functional spaces, 
    in particular Besov and Lorentz spaces, 
    that we shall use and  state some of their useful  properties. 
    We also study the operator ${\partial_{r} \over r } \Delta^{-1}$ in Proposition \ref{prop1}.
    Next, in section \ref{sectioncom}, we
     study the commutator
      $ \big[ {\partial_{r} \over r } \Delta^{-1}, v \cdot \nabla \big] 
       $.
 In section \ref{sectionapriori}, we turn to the proof of  a priori estimates for  sufficiently smooth
  solutions of \eqref{bsintro}. We first  prove in Proposition \ref{Energy}
   some basic energy estimates, next, we study the moments of  $\rho$ in Proposition \ref{prop07} 
    and then  we control $\|\zeta\|_{L^{3, 1}}$ in Proposition \ref{Strong}.  Lipschitz and Sobolev
     estimates 
     are  finally obtained in Proposition \ref{vorticity} and Proposition \ref{Soblev-est}.
      In section \ref{sectionproof}, we give the proof
         of Theorem \ref{thm1}: we obtain the existence part by using the a priori estimates
          and an approximation argument and then we prove the uniqueness part.
      Finally, the appendix \ref{appendix} is devoted to the proof of a priori estimates for
       convection-diffusion equations by the  Nash-De Giorgi iterations which are needed
        in the  estimate of the moments of $\rho$.  In the  appendix \ref{appendixB} we give the proof of Lemma \ref{commu} which is  a technical commutator lemma used in several places.

\section{Preliminaries}
\label{prelim1}
\subsection{Dyadic decomposition and functional spaces}
  Throughout this paper, $C$ stands for some real positive constant which may be different
  in each occurrence and $C_0$ denotes a positive number  depending on the initial data only.   We shall
  sometimes alternatively use the notation $X\lesssim Y$ for  the inequality  $X\leq CY$.
  
  When $B$ is a Banach space, we shall use the  shorthand  $L^p_{T}(B)$ for
   $L^p(0, T, B)$.
  
%
Now to introduce Besov spaces which are a generalization of Sobolev spaces we need to recall 
  the dyadic decomposition of the unity in  the whole space (see  \cite{che1}).
  \begin{prop}
There exist two positive radial  functions  $\chi\in \mathcal{D}(\mathbb R^3)$ and
$\varphi\in\mathcal{D}(\mathbb R^3\backslash{\{0\}})$ such that
\begin{enumerate}
\item
$\displaystyle{\chi(\xi)
+\sum_{q\in\NN}\varphi(2^{-q}\xi)=1},\quad \frac13\le \chi^2(\xi)+\sum_{q\in\NN}\varphi^2(2^{-q}\xi)\le 1$\;
$\forall\xi\in\mathbb R^3,$
\item
$ \textnormal{supp }\varphi(2^{-p}\cdot)\cap
\textnormal{supp }\varphi(2^{-q}\cdot)=\varnothing,$ if  $|p-q|\geq 2$,\\

\item
$\displaystyle{q\geq1\Rightarrow \textnormal{supp}\chi\cap
\textnormal{supp }\varphi(2^{-q})=\varnothing}$.
\end{enumerate}
\end{prop}
For every $u\in{\mathcal S}'(\mathbb R^3)$ we define the nonhomogeneous Littlewood-Paley operators by,
$$
\Delta_{-1}u=\chi(\hbox{D})u;\, \forall
q\in\mathbb N,\;\Delta_qu=\varphi(2^{-q}\hbox{D})u\; \quad\hbox{and}\quad
S_qu=\sum_{-1\leq j\leq q-1}\Delta_{j}u.
$$
One can easily prove that for every tempered distribution $u,$
\begin{equation}\label{dr2}
u=\sum_{q\geq -1}\Delta_q\,u.
\end{equation}

In the sequel we will  frequently use Bernstein inequalities (see for \mbox{example \cite{che1}}).
\begin{lem}\label{lb}\;
 There exists a constant $C$ such that for $k\in\NN$, \mbox{$1\leq a\leq b$}   and $u\in L^a$, we have
\begin{eqnarray*}
\sup_{|\alpha|=k}\|\partial ^{\alpha}S_{q}u\|_{L^b}
&\leq&
C^k\,2^{q(k+3(\frac{1}{a}-\frac{1}{b}))}\|S_{q}u\|_{L^a},
\end{eqnarray*}
and for $q\in\NN$
\begin{eqnarray*}
\ C^{-k}2^
{qk}\|{\Delta}_{q}u\|_{L^a}
&\leq&
\sup_{|\alpha|=k}\|\partial ^{\alpha}{\Delta}_{q}u\|_{L^a}
\leq
C^k2^{qk}\|{\Delta}_{q} u\|_{L^a}.
\end{eqnarray*}
\end{lem}

The basic tool of  the paradifferential calculus  is  Bony's
decomposition  \cite{b}.   It distinguishes in   a product
$uv$  three parts as follows:
$$
uv=T_u v+T_v u+\mathcal{R}(u,v),
$$
where
\begin{eqnarray*}
T_u v=\sum_{q}S_{q-1}u\Delta_q v, \quad\hbox{and}\quad \mathcal{R}(u,v)=
\sum_{q}\Delta_qu \widetilde \Delta_{q}v,
\end{eqnarray*}
$$
\textnormal{with}\quad {\widetilde \Delta}_{q}=\sum_{i=-1}^{1}\Delta_{q+i}.
$$
The term $T_{u}v$ is called the paraproduct of $v$ by $u$ and  $\mathcal{R}(u,v)$ the remainder term.
 The main interest of the paraproduct term is   that each term $S_{q-1} u \Delta_{q} v$
  has the support of its Fourier transform still  localized in  an annulus of size $2^q$ and thus $T_uv$ is a sum of almost orthogonal  functions.

 Let $(p,r)\in[1,+\infty]^2$ and $s\in\mathbb R,$ then the nonhomogeneous  Besov
\mbox{space $B_{p,r}^s$} is
the set of tempered distributions $u$ such that
$$
\|u\|_{B_{p,r}^s}:=\Big( 2^{qs}
\|\Delta_q u\|_{L^{p}}\Big)_{\ell^{r}}<+\infty.
$$
We remark that  the  Sobolev space $H^s$ coincides with  the  Besov space  $B_{2,2}^s$.
Also, by using the Bernstein inequalities  we get easily the embeddings
$$
B^s_{p_1,r_1}\hookrightarrow
B^{s+3({1\over p_2}-{1\over p_1})}_{p_2,r_2}, \qquad p_1\leq p_2\quad and \quad  r_1\leq r_2.
$$

Finally, let us notice that we can also characterize $L^p$ spaces in terms of  the 
 dyadic decomposition, see \cite{Stein}. For $p\in ]1, + \infty[$, 
  there exists $C>0$ such that: 
 $f$  belongs to $L^p$ if and only
  if $(\Delta_{q} f)_{q \geq -1} \in L^p l^2$  and
  \begin{equation}
  \label{LpB} C^{-1}  \Big\| \Big(  \sum_{q \geq -1} |\Delta_{q} f |^2 \Big)^{1 \over 2} \Big\|_{L^p}
   \leq \|f \|_{L^p}  \leq C 
 \Big\| \Big(  \sum_{q \geq -1} |\Delta_{q} f |^2 \Big)^{1 \over 2} \Big\|_{L^p}.
 \end{equation}
  
\subsection{Lorentz spaces and interpolation}

For $p\in]1,\infty[, q\in [1,+\infty],$ the Lorentz space $L^{p,q}$ can be defined by 
real interpolation from Lebesgue spaces:
$$
(L^{p_0},L^{p_1})_{(\theta,q)}=L^{p,q},
$$
where
$
1\leq p_0<p<p_1\leq\infty,
$
$\theta$ satisfies ${1\over p}={1-\theta\over p_0}
+{\theta\over p_1}$ and $1\leq q\leq\infty$.

From this definition, we get: 
\begin{equation}
\label{propLp1}
L^{p,q}\hookrightarrow L^{p,q'}, \quad L^{p,p}=L^p
\end{equation}
for every  $ 1< p<\infty, 1\leq q\leq q'\leq \infty$.

Lorentz spaces will arise in a natural way in our problem because of the following classical
  convolution results, for the proof see for instance \cite{Lemar,Neil}.
  \begin{Theo}
  \label{convolution}
  For every  $\alpha$,  $0<\alpha<d$,  $p_{i}\in  ]1, +\infty[$, $q_{i} \in  [1, + \infty]$,  such that
   $1+ {1 \over p_{1}} =  {1 \over p_{2}} + {1 \over p_{3} }$ and
    $ {1 \over q_{1}} = {1 \over q_{2}} + {1 \over q_{3}}$, 
      there exists $C>0$ such that
      \begin{equation}
      \label{convolution1}
    \| f * g \|_{L^{p_{1}, q_{1}}} \leq C \| f\|_{L^{p_{2}, q_{2}}} \, \|g \|_{L^{p_{3}, q_{3}}}.
    \end{equation}
    Moveover, in the case that  $p_{1}= \infty$, we have
   \begin{equation}
   \label{convolution2}
   \|f * g\|_{L^\infty(\mathbb{R}^d)}\le C\|f\|_{L^{\frac d\alpha,\infty}(\mathbb{R}^d)}\|g\|_{L^{\frac{d}{d-\alpha},1}
(\mathbb{R}^d)}.
\end{equation}
 \end{Theo}
 In particular,  by using this result   and  the fact that
  $1/ |x|^2$ belongs to $L^{ {3 \over 2}, \infty}(\mathbb{R}^3)$, we
  have that
  \begin{equation}
  \label{ndelta-1}
   \| \nabla \Delta^{-1} f \|_{L^\infty(\mathbb{R}^3)} \lesssim \| f \|_{L^{3, 1}(\mathbb{R}^3)}
   \end{equation}
  and thanks to the pointwise estimate \eqref{pointwise} that 
   \begin{equation}
 \label{convol}
\Big\|{v_{r}\over r }\Big\|_{L^\infty} \lesssim \| \zeta \|_{L^{3, 1} } .
\end{equation}
  To establish some functional inequalities involving Lorentz spaces  the following
    classical  interpolation result (see  \cite{Lemar} for example) will be very useful. 
\begin{Theo}\label{interpol}
\item Let $1\le p_1<p_2\le\infty, 1\le r_1<r_2\le\infty, q\in[1,\infty]$ and  $T$ be a linear bounded operator from $L^{p_i}$ to $L^{r_i}$. Let $\theta\in]0,1[$ and $p,r$ such that  $\frac1p=\frac{\theta}{p_1}+\frac{1-\theta}{p_2}$ and $ \frac1r=\frac{\theta}{r_1}+\frac{1-\theta}{r_2}$. Then $T$ is also bounded from $L^{p,q}$ to $L^{r,q}$ with
$$
\|T\|_{\mathcal{L}(L^{p,q};L^{r,q})}\le C \|T\|_{\mathcal{L}(L^{p_1};L^{r_1})}^\theta\|T\|_{\mathcal{L}(L^{p_2};L^{r_2})}^{1-\theta}.
$$
\end{Theo}

As a consequence, we obtain the following results.
\begin{prop}
\label{properties}
 For  $1<p<+\infty,$   $q \in [1, +\infty] $, then exists a constant $C>0$ such that the following estimates hold true 
\begin{enumerate}
 \item 
$\|uv\|_{L^{p,q}}
\le
C\|u\|_{L^\infty}\|v\|_{L^{p,q}}, $
\item  
$
\|T_u v\|_{L^{p,q}}\le C\|u\|_{L^\infty}\|v\|_{L^{p,q}}.
$
\item  Let us define the Riesz transform  $\mathcal{R}_{ij}=\partial_{i}\partial_j\Delta^{-1}, i,j\in\{1,2\}$, 
then $$
\|\rz u\|_{L^{p,q}}\le C\| u\|_{L^{p,q}}.
$$
\item For $s>\frac12$ we have $H^s\hookrightarrow L^{3,1}.$

\end{enumerate}

\end{prop}

\begin{proof} $\,$ \\ 
{\bf $(1)$} For a fixed function $u\in L^\infty$, the linear operator
 $T: \, v \mapsto uv$  belongs to $\mathcal{L}(L^p, L^p)$ with norm smaller that $\|u\|_{L^\infty}$
  and hence the result follows by interpolation from Theorem \ref{interpol}.
  
  {\bf $(3)$} In a similar way, for every $p \in ]1, + \infty[,$  $\mathcal{R}_{ij}\in \mathcal{L}(L^p, L^p)$
   thanks to the Calder\'on-Zygmund theorem and hence (3) follows again by using Theorem 
   \ref{interpol}.
     
{\bf$(2)$}  To establish the inequality, it  is again sufficient
 thanks to Theorem   \ref{interpol} to prove that for   $u\in L^\infty, v\in L^p$ we have
   $ \|T_u v \|_{L^p} \le C \|u \|_{L^\infty}\|v\|_{L^p}$.
  For this last purpose we will make use of the maximal functions tool. We will start with some  classical results in this subject. 
  For a  locally integrable function $f:\RR^3\to\RR,$ we  shall define its maximal function $\mathcal{M} f$ by
$$
\mathcal{M}f(x)=\sup_{r>0}\frac{1}{r^3}\int_{B(x,r)}|f(y)|dy.
$$
From the definition we get
\begin{equation}\label{lu01}
0\le\mathcal{M}(fg)(x)\le \|g\|_{L^\infty}\mathcal{M}f(x).
\end{equation}
It is well-known that $\mathcal{M}$ maps continuously $L^p$ to itself for $p\in]1,\infty].$
Moreover, we have the following  lemma. We refer to \cite{Stein} for a proof.
\begin{lem}\label{maximal}
\begin{enumerate}
\item Let $\psi\in \mathcal{S}(\RR^3)$ and define $\psi_\varepsilon(x)=\varepsilon^{-3}\psi(\varepsilon^{-1}x)$ for $\varepsilon>0$.
Then there exists $C>0$ such that for every $p\in [1,\infty]$, $\varepsilon \in (0, 1]$, we have 
$$
\sup_{\varepsilon>0}|\psi_\varepsilon\star f(x)|\le C\mathcal{M}f(x).
$$
In particular, we have
$$
\sup_{q\geq -1}|\Delta_q f(x)|\le C\mathcal{M}f(x).
$$
\item
Let $p\in]1,\infty]$ and $\{f_q, q\geq-1\}$ be a sequence belonging to $L^p\ell^2$. Then we have
$$
\Big\|\Big(\sum_q\big|\mathcal{M}f_q(x)\big|^2\Big)^{\frac12}\Big\|_{L^p}\le C\Big\|\Big(\sum_q\big|f_q(x)\big|^2\Big)^{\frac12}\Big\|_{L^p}.
$$
\end{enumerate}
\end{lem}

Let us now come back to the proof of $(2)$.
By using \eqref{LpB}, we have
    \begin{eqnarray*}
\|T_u v\|_{L^p}&\lesssim&\Big\|\Big(\sum_{j\geq -1}|\Delta_j(T_u v)|^2\Big)^{\frac12}\Big\|_{L^p}
=\Big\|\Big(\sum_{j\geq -1}\Big|\Delta_j\big(\sum_{|j-q|\le4}S_{q-1}u\Delta_q v\big)\Big|^2\Big)^{\frac12}\Big\|_{L^p}
\end{eqnarray*}
This yields according  to Lemma \ref{maximal} and \eqref{lu01},
\begin{eqnarray*}
\|T_u v\|_{L^p}&\lesssim&\Big\|\Big(\sum_{j\geq -1}\Big(\sum_{|j-q|\le4}\mathcal{M}(S_{q-1}u\Delta_q v)\Big)^2\Big)^{\frac12}\Big\|_{L^p}\lesssim \|u\|_{L^\infty}\Big\|\Big(\sum_{j\geq -1}\Big(\sum_{|j-q|\le4}\mathcal{M}\,\Delta_q v\Big)^2\Big)^{\frac12}\Big\|_{L^p}\\
&\lesssim&\|u\|_{L^\infty}\Big\|\Big(\sum_{q\geq-1}\big(\mathcal{M}\Delta_q v\big)^2\Big)^{\frac12}\Big\|_{L^p}
\lesssim\|u\|_{L^\infty}\Big\|\Big(\sum_{q\geq-1}\big(\Delta_q v\big)^2\Big)^{\frac12}\Big\|_{L^p}\\
&\lesssim&\|u\|_{L^\infty}\|v\|_{L^p}
\end{eqnarray*}
where the last estimate follows from  a new use of  \eqref{LpB}. This ends the proof of (2).

{\bf($4$)} This  embedding    follows  from  Sobolev embeddings 
 combined with Theorem \ref{interpol}. This is left to the reader.
\end{proof}

\subsection{Some useful  commutator estimates }
This section is devoted to the study of  some  basic commutators which will be needed
 in our main commutator estimates, especially in Theorem \ref{prop2} and Proposition \ref{propcomm} . 
Our first result reads as follows. The proof is postponed    to   Appendix \ref{appendixB}. 
 \begin{lem}
 \label{commu}Given $(p,r,\rho,m)\in[1,+\infty]^4$ such that  
 $$1+\frac1p=\frac1m+\frac1\rho+\frac1r, \quad p\geq r\quad\hbox{and}\quad\rho>3(1-\frac1r).
 $$ Let $f,g$ and $h$ be three functions such that $\nabla f\in L^\rho, g\in L^m$ and $x\,\mathcal{F}^{-1}h\in L^{r}$. Then
 $$
 \big\|\big[h(\textnormal{D}),f \big]g\big\|_{L^p}\leq C \|x\mathcal{F}^{-1}h\|_{L^{r}}\|\nabla f\|_{L^\rho}\|g\|_{L^{m}}.
 $$
where $C$ is a constant.
 \end{lem}
 
As an application of Lemma \ref{commu} we get the following commutator estimates.
\begin{lem}\label{lemcom}
Let $p,m,\rho\in[1,+\infty]$ such that $\frac1p=\frac1m+\frac1\rho.$ Then, there exists $C>0$ such that  for $\nabla f\in L^\rho, g\in L^m$ and for every $q\in \NN\cup\{-1\}$
$$
\|[\Delta_q,f]g\|_{\dot{W}^{1,p}}\le C\|\nabla f\|_{L^\rho}\|g\|_{L^m},
$$
with the following definition $\|\varphi\|_{\dot{W}^{1,p}}=\|\nabla\varphi\|_{L^p}.$
\end{lem}
\begin{proof}
We write for $i=1,2,3,$
\begin{eqnarray*}
\partial_i\big( [\Delta_q,f]g\big)&=&[\partial_i\Delta_q,f]g-\partial_if\Delta_qg
=[h_q(\hbox{D}),f]g-\partial_if\Delta_qg,
\end{eqnarray*}
with $ h_q(\xi)=2^q\phi(2^{-q}\xi),$ and $\phi\in\mathcal{S}(\RR^3).$ Using Lemma \ref{commu} we get
\begin{eqnarray*}
\|[h_q(\hbox{D}),f]g\|_{L^p}&\le& C \|x\mathcal{F}^{-1}h_q\|_{L^1}\|\nabla f\|_{L^\rho}\|g\|_{L^m}
\le C\|\nabla f\|_{L^\rho}\|g\|_{L^m}.
\end{eqnarray*}
For the other term, the  H\"{o}lder inequality yields
\begin{eqnarray*}
\|\partial_if\Delta_qg\|_{L^p}&\le&C\|\nabla f\|_{L^\rho}\|\Delta_qg\|_{L^m}\\
&\le&C\|\nabla f\|_{L^\rho}\|g\|_{L^m}.
\end{eqnarray*}
\end{proof}
\subsection{Some algebraic identities}
We intend in this paragraph to describe first the action of the  operator $\frac{\partial_r}{r}\Delta^{-1}u$ over axisymmetric functions. We will show that it behaves like Riesz transforms. The second part is concerned with   the study of some algebraic identities involving  some multipliers which will appear in a natural way when try to study our main   commutator $\big[{\partial_r}/{r})\Delta^{-1},v\cdot\nabla\big]\rho.$
\begin{prop}\label{prop1}
We have for  every axisymmetric smooth scalar function $u$
\begin{equation}
\label{prop1-1}
({\partial_r}/{r})\Delta^{-1}u(x)=\frac{x_2^2}{r^2}\mathcal{R}_{11}u(x)+\frac{x_1^2}{r^2}\mathcal{R}_{22}u(x)-2\frac{x_1x_2}{r^2}\mathcal{R}_{12}u(x),
\end{equation}
with 
$
\mathcal{R}_{ij}=\partial_{ij}\Delta^{-1}.$ Moreover, for $p\in]1,\infty[, q\in[1,\infty]$ there exists $C>0$ such that
\begin{equation}
\label{prop1-2}
\|({\partial_r}/{r})\Delta^{-1}u\|_{L^{p,q}}\le C\| u\|_{L^{p,q}}.
\end{equation}

\end{prop}
\begin{proof}
We set $f=\Delta^{-1}u,$ then we can show from Biot-Savart law that $f$ is also axisymmetric. Hence we get
 by using  polar coordinates that 
\begin{equation}
\label{laplace1}
\partial_{11}f+\partial_{22}f  =  (\partial_r/r)f+\partial_{rr}f
\end{equation}
 where 
$$
\partial_r=\frac{x_1}{r}\partial_1+\frac{x_2}{r}\partial_2.
$$
By using this expression of $\partial_{r}$, we obtain
\begin{eqnarray*}
\partial_{rr}&=&\big(\frac{x_1}{r}\partial_1+\frac{x_2}{r}\partial_2\big)^2
=\partial_r(\frac{x_1}{r})\partial_1+\partial_r(\frac{x_2}{r})\partial_2+\frac{x_1^2}{r^2}\partial_{11}+\frac{x_2^2}{r^2}\partial_{22}
+\frac{2x_1 x_2}{r^2}\partial_{12}.\\
&=&\frac{x_1^2}{r^2}\partial_{11}+\frac{x_2^2}{r^2}\partial_{22}+\frac{2x_1 x_2}{r^2}\partial_{12}
\end{eqnarray*}
 since
 $$
\partial_r(\frac{x_i}{r})=0,\quad\forall i\in\{1,2\}.
$$
This yields by using \eqref{laplace1} that 
\begin{eqnarray*}
{\partial_r \over r} f &=&(1-\frac{x_1^2}{r^2})\partial_{11} f +(1-\frac{x_2^2}{r^2})\partial_{22} f -\frac{2x_1 x_2}{r^2}\partial_{12}f \\
&=&\frac{x_2^2}{r^2}\partial_{11} f +\frac{x_1^2}{r^2}\partial_{22}f -\frac{2x_1 x_2}{r^2}\partial_{12}f.
\end{eqnarray*}
 To  get \eqref{prop1-1},   it suffices replace $f$ by  $\Delta^{-1} u $. 

The estimate  \eqref{prop1-2} is a consequence of \eqref{prop1-1} and the estimates (1) and  (3)
 of Proposition \eqref{properties} since  for every $i,j\in\{1,2\},\,\frac{x_{i}x_{j}}{r^2} \in L^\infty.$
\end{proof}

We  shall also need  the following identities and estimates.
\begin{lem}\label{ident001}
 For every $f\in \mathcal{S}(\RR^3, \mathbb{R})$, we have
\begin{enumerate}
\item For $ i,j\in\{1,2,3\}$
$$
\Delta^{-1}(x_i\partial_j f)=x_i\partial_j\Delta^{-1} f
+ \mathcal{L}_{ij} f$$
 where $\mathcal{L}_{ij} f = 
-2\mathcal{R}_{ij}\Delta^{-1}$f.
 Moreover, we have the estimates: 
\begin{eqnarray}
\label{Lijinfty} & & 
 \| \nabla \mathcal{L}_{ij} f \|_{L^\infty} \leq C \| f \|_{L^{3, 1}},  \\
 \label{Lij3}  & & \| \nabla^2 \mathcal{L}_{ij} f\|_{L^{p,q}} \leq C \| f \|_{L^{p, q}}, 
  \quad
  p \in]1, +\infty [, \, q \in [1, +\infty]
  \end{eqnarray}
\item  For $i,j,k\in\{1,2,3\}$
$$
\rz(x_k f)=x_k\rz f+\mathcal{L}_{ij}^kf,
$$
with
$$
\mathcal{L}_{ij}^k:=-2\partial_k\Delta^{-1}\rz+\delta_{ik}\partial_j\Delta^{-1}+\delta_{jk}\partial_i\Delta^{-1}$$
where $\delta_{ij}$ denotes  the Kronecker symbol.
Moreover we have the estimates 
\begin{eqnarray}
\label{Lijkinfty}   & & \| \mathcal{L}_{ij}^k  f \|_{L^\infty} \leq C  \|f\|_{L^{3, 1}}, \\
\label{Lijk3} & & \| \nabla \mathcal{L}_{ij}^k f\|_{L^{p,q}} \leq C \|  f \|_{L^{p,q}}, \quad
  p \in]1, +\infty [, \, q \in [1, +\infty].
 \end{eqnarray}
\end{enumerate}
\end{lem}
\begin{proof}
{\bf $(1)$}  We  first expand
\begin{eqnarray*}
\Delta(x_i\partial_j\Delta^{-1} f-2\mathcal{R}_{ij}f) & =&2\partial_{ij}\Delta^{-1}f+x_i\partial_j f-2\mathcal{R}_{ij}f\\
&=&x_j\partial_j f
\end{eqnarray*}
This  yields 
$$
\Delta^{-1}(x_j\partial_j f)=x_i\partial_j\Delta^{-1} f-2\mathcal{R}_{ij}f+P(x),
$$
with  $P$ a harmonic polynomial. We can easily see that the r.h.s of this identity and $\mathcal{R}_{ij}f$ are decreasing at infinity. Thus to prove that $P$ is zero it suffices to prove that $x_i\partial_j\Delta^{-1} f$ goes to zero at infinity.  Since 
\begin{eqnarray*}
\big|x_i\partial_j\Delta^{-1} f\big|&\lesssim& |x_j|\int_{\RR^3}\frac{|f(y)|}{|x-y|^2}dy
\lesssim \int_{\RR^3}\frac{|f(y)|}{|x-y|}dy+\int_{\RR^3}\frac{|y_jf(y)|}{|x-y|^2}dy.
\end{eqnarray*}
Using Proposition \ref{convolution}, we get  that $x_i\partial_j\Delta^{-1} f\in L^p$, for every $p>3.$ Hence we get $P=0.$

The estimates \eqref{Lijinfty}, \eqref{Lij3} are a direct consequence of the above expression 
  and \eqref{ndelta-1} and the estimate (3) of Proposition \ref{properties}.

{\bf($2$)}   We use the same  idea as previously. We first  get the identity
\begin{eqnarray*} \Delta \mathcal{L}_{ij}^k f &= & 
  \Delta \big( \mathcal{R}_{ij}(x_{k} f) - x_{k} \mathcal{R}_{ij} f \big)
   =  \partial_{ij}(x_{k} f ) -  2\partial_{k} \mathcal{R}_{ij} f  - x_{k} \partial_{ij} f \\
    &  = &  \delta_{ik} \partial_{j} f + \delta_{jk} \partial_{i} f - 2 \partial_{k} \mathcal{R}_{ij}f
\end{eqnarray*}
and by the same argument as above, we  finally obtain that 
$$   \mathcal{R}_{ij}(x_{k} f) - x_{k} \mathcal{R}_{ij} f 
 =   \delta_{ik} \partial_{j} \Delta^{-1} f + \delta_{jk} \partial_{i}\Delta^{-1} f - 2 \partial_{k} \mathcal{R}_{ij}
 \Delta^{-1}f.$$
 The estimates \eqref{Lijkinfty}, \eqref{Lijk3} are a direct consequence of the above expression 
  and \eqref{ndelta-1} and the estimate (3) of Proposition \ref{properties}.
\end{proof}

\section{Commutator estimates}
\label{sectioncom}

\subsection{The commutator between   the advection operator  and   ${\partial_{r } \over r } \Delta^{-1}$}
\label{subsec1}
In this part we discuss the commutation between the operators $\frac{\partial_r}{r}\Delta^{-1}$ and $v\cdot\nabla.$
This is a crucial estimate in order to  get better a priori estimates for the solution
 of  \eqref{bsintro} by using our transformation. 
 Our result reads as follows.
\begin{Theo}\label{prop2}
Let $v$ be an axisymmetric smooth and divergence free without swirl  vector field  and $\rho$ an axisymmetric smooth scalar function. Then we have,  with the notation $x_h=(x_1,x_2),$ that 
$$
\Big\|\big[({\partial_r}/{r})\Delta^{-1},v\cdot\nabla\big]\rho\Big\|_{L^{3,1}}\lesssim\|\omega_{\theta}/r\|_{L^{3,1}}\big(\|\rho x_h\|_{B_{\infty,1}^0\cap L^2}+\|\rho\|_{B_{2,1}^{\frac12}}\big).$$
\end{Theo}

\begin{proof}
Since the functions $\rho$ and $v\cdot\nabla\rho$ are axisymmetric then using the identity of   Proposition \ref{prop1} we have 
$$
({\partial_r}/{r})\Delta^{-1}\rho(x)=\frac{x_2^2}{r^2}\mathcal{R}_{11}\rho(x)+\frac{x_1^2}{r^2}\mathcal{R}_{22}\rho(x)-2\frac{x_1x_2}{r^2}\mathcal{R}_{12}\rho(x):=
\sum_{i,j=1}^2a_{ij}(x)\mathcal{R}_{i,j}\rho(x)
$$
and also 
$$
({\partial_r}/{r})\Delta^{-1}(v\cdot\nabla\rho)(x)=\sum_{i,j=1}^2a_{ij}(x)\mathcal{R}_{i,j}(v\cdot\nabla\rho)(x).
$$
Since $v$ has no  swirl  and the functions $a_{i,j}$ do not depend on $r$ and $z$, we  have for every $1\le i,j\le2$
\begin{eqnarray*}
v\cdot\nabla a_{i,j}(x)&=&v^r\partial_r a_{i,j}+v^z\partial_3 a_{i,j}= 0.
\end{eqnarray*}
Consequently our  commutator can be rewritten as 
$$
\big[({\partial_r}/{r})\Delta^{-1},v\cdot\nabla\big]\rho(x)=\sum_{i,j=1}^2a_{i,j}(x)[\mathcal{R}_{ij},v\cdot\nabla]\rho\\
=\sum_{i,j=1}^2a_{i,j}(x)\textnormal{div}\{[\mathcal{R}_{ij},v]\rho\}
$$
where we have used the fact that $v$ is divergence free to get the last equality.
 By using that $a_{ij} \in L^\infty$ and the estimate (1) of Proposition \ref{properties},   we first obtain that 
 \begin{equation}
 \label{com1}
\Big\|\big[({\partial_r}/{r})\Delta^{-1},v\cdot\nabla\big]\rho\Big\|_{L^{3,1}}\le\sum_{i,j=1}^2\Big\|\textnormal{div}
 \Big([\mathcal{R}_{ij},v]\rho\Big)\Big\|_{L^{3,1}}.
\end{equation}
The terms $\partial_1\big([\mathcal{R}_{ij},v^1]\rho\big)$ and $\partial_2\big([\mathcal{R}_{ij},v^2]\rho\big)$ 
 can be  treated  in  same way and hence, we  shall prove   the estimate of the first one only.  The estimate
 of $\partial_{3}  \big([\mathcal{R}_{ij},v^1]\rho\big)$  which is easier will be done in a second step. 

$\bullet$ {\it Estimate of $\partial_1\big( [\mathcal{R}_{ij},v^1]\rho\big)$}.
 Since $v$ is divergence free,  we have that   $\Delta v= - \nabla\wedge \omega.$ 
  Hence for axisymmetric flows  (where  in particular $\omega=\omega_\theta e_\theta$), we obtain that 
\begin{eqnarray*}
v^1(x)&=&\Delta^{-1}\partial_3\omega^2
=\Delta^{-1}\partial_3(x_1({\omega_\theta}/{r})).
\end{eqnarray*}
Applying Lemma \ref{ident001}-(1) we get
\begin{equation}\label{ident1}
v^1(x)= x_1\Delta^{-1}\partial_3(\omega_\theta/r)+\mathcal{L}(\omega_\theta/r),\quad\hbox{with}\quad \mathcal{L}= - 2\partial_{13}\Delta^{-2}
\end{equation}
(we omit the subscript $ij$ for notational convenience).
Consequently the commutator can be rewritten under the form
\begin{eqnarray}\label{decompo}
\nonumber & &\partial_1\{[\mathcal{R}_{ij},v^1]\rho\}= \partial_1\Big(\big[\mathcal{R}_{i,j},\mathcal{L}(\omega_\theta/r)\big]\rho \Big)
 +   \partial_1\Big(\big[\mathcal{R}_{i,j},\big(\Delta^{-1} \partial_{3}(\omega_\theta/r) \big) x_{1}\big]\rho \Big) \\
 & &\nonumber  =   \partial_1\Big(\big[\mathcal{R}_{i,j},\mathcal{L}(\omega_\theta/r)\big]\rho \Big) + 
  \partial_1\Big(\big[\mathcal{R}_{i,j},\big(\Delta^{-1} \partial_{3}(\omega_\theta/r) \big) \big] x_{1}\rho \Big)
 +
   \partial_1\Big(\big(\Delta^{-1} \partial_{3}(\omega_\theta/r) \big) \big[ \mathcal{R}_{ij},  x_{1}\big]\rho \Big) \\ 
\nonumber& & =   \partial_1\Big(\big(\Delta^{-1} \partial_{3}(\omega_\theta/r) \big) \mathcal{L}_{ij}^1\rho \Big)
 +    \partial_1\Big(\big[\mathcal{R}_{i,j},\mathcal{L}(\omega_\theta/r)\big]\rho \Big) +
  \partial_1\Big(\big[\mathcal{R}_{i,j},\big(\Delta^{-1} \partial_{3}(\omega_\theta/r) \big) \big] x_{1}\rho \Big) \\
& & =\hbox{I}+\hbox{II}+\hbox{III}
\end{eqnarray}
where we have used the identity  (2) of Lemma \ref{ident001}.

\underline{{\it{Estimate of }}$\hbox{I}$}.
 We write
\begin{equation}\label{Eq11}
\partial_1\Big(\partial_3\Delta^{-1}(\omega_{\theta}/r)\,\mathcal{L}_{ij}^1\rho\Big)=\mathcal{R}_{13}(\omega_{\theta}/r)\,\mathcal{L}_{ij}^1\rho+\partial_3\Delta^{-1}(\omega_{\theta}/r)\,\partial_1\mathcal{L}_{ij}^1\rho.
\end{equation}
By using   (1) and (3) of   Proposition \ref{properties} and  \eqref{Lijkinfty} we have
\begin{eqnarray*}
\|\mathcal{R}_{13}(\omega_{\theta}/r)\,\mathcal{L}_{ij}^1\rho\|_{L^{3,1}}&\le&\|\mathcal{R}_{13}(\omega_{\theta}/r)\|_{L^{3,1}}\,\|\mathcal{L}_{ij}^1\rho\|_{L^{\infty}}
\le C\|\omega_{\theta}/r\|_{L^{3,1}}\|\rho\|_{L^{3,1}}
\end{eqnarray*}
and by using   Proposition \ref{properties}-(1) and \eqref{ndelta-1},  \eqref{Lijk3}, we also obtain
\begin{eqnarray*}
\|\partial_3\Delta^{-1}(\omega_{\theta}/r)\,\partial_1\mathcal{L}_{ij}^1\rho\|_{L^{3,1}}&\le&\|\partial_{3}\Delta^{-1}(\omega_{\theta}/r)\|_{L^{\infty}}\,\|\partial_1\mathcal{L}_{ij}^1\rho\|_{L^{3,1}}
\le C\|\omega_{\theta}/r\|_{L^{3,1}}\|\rho\|_{L^{3,1}}.
\end{eqnarray*}
Combining these estimates we find
\begin{equation}\label{c1}
\|\hbox{I}\|_{L^{3,1}}\le C\|\omega_{\theta}/r\|_{L^{3,1}}\|\rho\|_{L^{3,1}}.
\end{equation}
\underline{{\it{Estimate of }}$\hbox{II}$}. We will use Bony decomposition 
$$
\hbox{II}=\hbox{II}_1+\hbox{II}_2+\hbox{II}_3,
$$
with
\begin{eqnarray*}
\hbox{II}_1&=&\partial_1\sum_{q\geq 0}[\rz, S_{q-1}(\mathcal{L}(\omega_{\theta}/r))]\Delta_q\rho\\
\hbox{II}_2&=&\partial_1\sum_{q\geq 0}[\rz, \Delta_{q}(\mathcal{L}(\omega_{\theta}/r))]S_{q-1}\rho\\
\hbox{II}_3&=&\partial_1\sum_{q\geq-1}[\rz, \Delta_{q}(\mathcal{L}(\omega_{\theta}/r))]\tilde\Delta_q\rho.
\end{eqnarray*}
For the first term we easily get  that there exists a function $\psi\in \mathcal{S}(\RR^3)$ such that
$$
\hbox{II}_1=\sum_{q\geq 0}\partial_1\big\{\big[\psi_q(\hbox{D}), S_{q-1}(\mathcal{L}(\omega_{\theta}/r))\big]\Delta_q\rho\big\},
$$
with $\psi_q=2^{3q}\psi(2^q\cdot).$  By using the  Bernstein inequality, this yields 
\begin{eqnarray*}
\Big\|\partial_1\big\{\big[\psi_q(\hbox{D}), S_{q-1}(\mathcal{L}(\omega_{\theta}/r))\big]\Delta_q\rho\big\}\Big\|_{L^{2}}&\le& C2^{q}\Big\|\big[\psi_q(\hbox{D}), S_{q-1}( \mathcal{L}(\omega_{\theta}/r))\big]\Delta_q\rho\Big\|_{L^{2}}
\end{eqnarray*}
Thanks to  Lemma \ref{commu} and  \eqref{Lijinfty}, we find
\begin{eqnarray*}
\Big\|\partial_1\big\{\big[\psi_q(\hbox{D}), S_{q-1}(\mathcal{L}(\omega_{\theta}/r))\big]\Delta_q\rho\big\}\Big\|_{L^{2}}&\le& C2^q\|x\psi_q\|_{L^1}\|\nabla\mathcal{L}(\omega_\theta/r)\|_{L^\infty}\|\Delta_q\rho\|_{L^2}\\
&\le& C\|x\psi\|_{L^1}\|\omega_\theta/r\|_{L^{3,1}}\|\Delta_q\rho\|_{L^2}.
\end{eqnarray*}
It follows that
\begin{eqnarray*}
\|\hbox{II}_1\|_{B_{2,1}^{\frac12}}&\le& C\sum_{q\in\NN}2^{q\frac12}
\Big\|\partial_1\big\{\big[\psi_q(\hbox{D}), S_{q-1}(\mathcal{L}(\omega_{\theta}/r))\big]\Delta_q\rho\big\}\Big\|_{L^{2}}\le  C\|\omega_\theta/r\|_{L^{3,1}}\|\rho\|_{B_{2,1}^{\frac12}}
\end{eqnarray*}
and hence by using the embedding  $B_{2,1}^{\frac12}\hookrightarrow L^{3,1}$  (see  Proposition \ref{properties}- (4)), we obtain
\begin{eqnarray}\label{com0001}
\|\hbox{II}_1\|_{L^{3,1}}&\le& C\|\omega_\theta/r\|_{L^{3,1}}\|\rho\|_{B_{2,1}^{\frac12}}.
\end{eqnarray}

To estimate the term $\hbox{II}_2$  we do not need to detect cancellation in   the structure of the commutator,
 we just write
\begin{eqnarray*}
\hbox{II}_2&=&\sum_{q\geq 0}\partial_1\rz\Big( \Delta_{q}(\mathcal{L}(\omega_{\theta}/r))S_{q-1}\rho\Big)-\sum_{q\geq0}\partial_1\big\{ \Delta_{q}(\mathcal{L}(\omega_{\theta}/r))\rz S_{q-1}\rho\big\}.
\end{eqnarray*}
A useful remark is that  
thanks to the  Bernstein inequalities and \eqref{Lij3}, we have
\begin{equation}
\label{Lmieux}
 \|\Delta_{q} \mathcal{L}f \|_{L^p} \lesssim 2^{- 2 q} \|\nabla^2 \mathcal{L} \Delta_{q}f\|_{L^p}
  \lesssim   2^{-2 q}  \|f \|_{L^p}, \quad \forall q \geq 0, \, p \in ]1, + \infty[.
\end{equation}
 This yields  by using the  H\"{o}lder inequality and  Proposition \ref{properties}-(3) that
\begin{eqnarray*}
\|\hbox{II}_2\|_{B_{2,1}^{\frac12}}&\lesssim&\sum_{q\geq 0}2^{\frac32q}\big\| \Delta_{q}(\mathcal{L}(\omega_{\theta}/r))S_{q-1}\rho\big\|_{L^2}+\sum_{q\geq0}2^{q\frac32}\big\|\Delta_{q}(\mathcal{L}(\omega_{\theta}/r))\rz S_{q-1}\rho\|_{L^2}\\
&\lesssim&\sum_{q\geq 0}2^{q\frac32}\big\| \Delta_{q}\mathcal{L}(\omega_{\theta}/r)\|_{L^3}\big(\|S_{q-1}\rho\big\|_{L^6}+\|\rz S_{q-1}\rho\big\|_{L^6}\big)\\
&\lesssim&\|\omega_\theta/r\|_{L^{3}}\sum_{q\geq 0}2^{-q\frac12}\|S_{q-1}\rho\big\|_{L^6}\\
&\lesssim&\|\omega_\theta/r\|_{L^{3}} \sum_{q \geq 0 }\sum_{ k\le q-2}2^{\frac12(k-q)}(2^{\frac k2}\|\Delta_k\rho\|_{L^2})\\
&\lesssim&\|\omega_\theta/r\|_{L^{3}}\|\rho\|_{B_{2,1}^{\frac12}}.
\end{eqnarray*}
Hence we get from Proposition \ref{properties}-(4) that 
\begin{eqnarray}\label{com0002}
\|\hbox{II}_2\|_{L^{3,1}}&\le& C\|\omega_\theta/r\|_{L^{3,1}}\|\rho\|_{B_{2,1}^{\frac12}}.
\end{eqnarray}
For the term $\hbox{II}_3$ we write  
\begin{eqnarray*}
\hbox{II}_3&=&\partial_1\sum_{q\geq1}[\rz, \Delta_{q}(\mathcal{L}(\omega_{\theta}/r))]\tilde\Delta_q\rho
+\partial_1\sum_{-1\le q\leq0}[\rz, \Delta_{q}(\mathcal{L}(\omega_{\theta}/r))]\tilde\Delta_q\rho\\
&:=&\hbox{II}_{31}+\hbox{II}_{32}.
\end{eqnarray*}
To estimate the first term we first use the  Bernstein inequality to get
\begin{eqnarray*}
\|\Delta_k\hbox{II}_{31}\|_{L^2}&\lesssim& 2^{k}\sum_{q\geq k-4}\big\|[\rz, \Delta_{q}(\mathcal{L}(\omega_{\theta}/r))]\tilde\Delta_q\rho\big\|_{L^2}.
\end{eqnarray*}
Next, to estimate the terms  inside the sum we do not need to  use the structure of the commutator. By
 using again the H\"{o}lder inequality,  \eqref{Lmieux} and the Bernstein inequality,  we obtain
\begin{eqnarray*}
\big\|[\rz, \Delta_{q}(\mathcal{L}(\omega_{\theta}/r))]\tilde\Delta_q\rho\big\|_{L^2}
&\lesssim& \big\|\Delta_{q}(\mathcal{L}(\omega_{\theta}/r))\|_{L^3}\|\tilde\Delta_q\rho\big\|_{L^6}
+ \|\Delta_{q}(\mathcal{L}(\omega_{\theta}/r))\|_{L^3}\|\rz\tilde\Delta_q\rho\|_{L^6}\\
&\lesssim&2^{-q} \|\omega_{\theta}/r\|_{L^3}\|\tilde\Delta_q\rho\|_{L^2}.
\end{eqnarray*}
It follows  by using again Proposition \ref{properties}-(4)  that
\begin{eqnarray*}
 \|\hbox{II}_{31}\|_{L^{3, 1}}\lesssim  \|\hbox{II}_{31}\|_{B_{2,1}^{\frac12}}\lesssim\|\omega_{\theta}/r\|_{L^3} \sum_{k \geq -1}\sum_{q\geq k-4}2^{\frac32(k-q)}2^{q\frac12}\|\tilde\Delta_q\rho\|_{L^3}
\lesssim \|\omega_{\theta}/r\|_{L^3}\|\rho\|_{B_{2,1}^{\frac12}}.
\end{eqnarray*}
For the estimate of the   low frequencies term  $\hbox{II}_{32}$ we need to use  more deeply  the structure of the commutator. We first write
$$
\hbox{II}_{32}=\sum_{-1\le q\le 0}[\partial_1\rz, \Delta_{q}(\mathcal{L}(\omega_{\theta}/r))]\tilde\Delta_q\rho-\sum_{-1\le q\le 0}\partial_1\mathcal{L}\Delta_q(\omega_\theta/r)\rz\tilde\Delta_q\rho.
$$
The last term of the above identity is estimated as follows by using again Proposition \ref{properties}
-(1) and (3) and \eqref{Lijinfty}
\begin{eqnarray*}
\big\|\sum_{-1\le q\le 0}\partial_1\mathcal{L}\Delta_q(\omega_\theta/r)\rz\tilde\Delta_q\rho\|_{B_{2,1}^{\frac12}}&\lesssim&\sum_{-1\le q\le 0}\|\partial_1\mathcal{L}\Delta_q(\omega_\theta/r)\rz\tilde\Delta_q\rho\|_{L^2}\\
&\lesssim&\|\partial_1\mathcal{L}(\omega_\theta/r)\|_{L^\infty}\|\rho\|_{L^2}\\
&\lesssim&\|\omega_\theta/r\|_{L^{3,1}}\|\rho\|_{B_{2,1}^{\frac12}}.
\end{eqnarray*}
To estimate the first term of $\hbox{II}_{32}$ we write  for every $-1\le q\le 0$ thanks  to Lemma \ref{commu}
 that $$
\Big\|[\partial_1\rz, \Delta_{q}(\mathcal{L}(\omega_{\theta}/r))]\tilde\Delta_q\rho\Big\|_{L^{\frac52}}\lesssim \|xh\|_{L^{{\frac{10}{9}}}}\|\nabla\mathcal{L}(\omega_{\theta}/r)\|_{L^\infty}\|\tilde\Delta_q\rho\|_{L^2}.
$$
where $\widehat h(\xi)= \xi_1\frac{\xi_i\xi_j}{|\xi|^2}\tilde\chi(\xi)$ and $\tilde\chi\in \mathcal{D}(\RR^3).$ Using Mikhlin-H\"{o}rmander Theorem we have
$$
|h(x)|\le C(1+|x|)^{-4},\,\forall x\in\RR^3.
$$
This gives in particular   $xh\in L^{\frac{10}{9}}.$ Therefore we get by using again \eqref{Lijinfty} that 
\begin{eqnarray*}
\big\|\sum_{-1\le q\le 0}[\partial_1\rz, \Delta_{q}(\mathcal{L}(\omega_{\theta}/r))]\tilde\Delta_q\rho\big\|_{B_{\frac52,1}^{\frac15}}\lesssim \|\nabla\mathcal{L}(\omega_{\theta}/r)\|_{L^\infty}\|\rho\|_{L^2}
\lesssim\|\omega_{\theta}/r\|_{L^{3,1}}\|\rho\|_{L^2}.
\end{eqnarray*}
By using the embedding $B_{\frac52,1}^{\frac15}\hookrightarrow L^{3,1}$ which comes from Proposition 
 \ref{properties}-(4), we find that 
\begin{eqnarray*}
\|\hbox{II}_{32}\|_{L^{3,1}}&\lesssim&\|\omega_{\theta}/r\|_{L^{3,1}}\|\rho\|_{L^2}.
\end{eqnarray*}
We have thus obtained  that  the term $\hbox{II}_{3}$ enjoys the estimate
$$ \|\hbox{II}_{3}\|_{L^{3,1}}\lesssim\|\omega_{\theta}/r\|_{L^{3,1}}\|\rho\|_{L^2}. $$

Consequently,  by gathering this last estimate and the  estimates  \eqref{com0001}, \eqref{com0002}, we finally get
 that 
\begin{equation}\label{eq0076}
\|\hbox{II}\|_{L^{3,1}}\le C\|\omega_{\theta}/r\|_{L^{3,1}}\|\rho\|_{B_{2,1}^{\frac12}}.
\end{equation}
\underline{{\it{Estimate of }}$\hbox{III}$.}
We  also decompose the term $\hbox{III}$  by using Bony's formula as follows:  
$$
\hbox{III}=\hbox{III}_1+\hbox{III}_2+\hbox{III}_3,
$$
with
\begin{eqnarray*}
\hbox{III}_1&=&\partial_1\sum_{q\geq 0}[\rz, S_{q-1}(\partial_3\Delta^{-1}(\omega_{\theta}/r))]\Delta_q(x_1\rho)\\
\hbox{III}_2&=&\partial_1\sum_{q\geq 0}[\rz, \Delta_{q}(\partial_3\Delta^{-1}(\omega_{\theta}/r))]S_{q-1}(x_1\rho)\\
\hbox{III}_3&=&\partial_1\sum_{q\geq-1}[\rz, \Delta_{q}(\partial_3\Delta^{-1}(\omega_{\theta}/r))]\tilde\Delta_q(x_1\rho).
\end{eqnarray*}
As we have done  to handle the  term $\hbox{II}_1,$  we can use that  there exists a function $\psi\in \mathcal{S}(\RR^3)$ such that
$$
\hbox{III}_1=\sum_{q\geq 0}\partial_1\Big([\psi_q(\hbox{D}), S_{q-1}(\partial_3\Delta^{-1}(\omega_{\theta}/r))]\Delta_q(x_1\rho)\Big),
$$
with $\psi_q=2^{3q}\psi(2^q\cdot).$ For every $p\in]1,\infty[$, we first write thanks to  the  Bernstein inequality 
that 
$$
\Big\|\partial_1\big\{\big[\psi_q(\hbox{D}), S_{q-1}(\partial_3\Delta^{-1}(\omega_{\theta}/r))\big]\Delta_q(x_1\rho)\big\}\Big\|_{L^{p}}\le C2^{q}\Big\|\big[\psi_q(\hbox{D}), S_{q-1}(\partial_3\Delta^{-1}(\omega_{\theta}/r))\big]\Delta_q(x_1\rho)\Big\|_{L^{p}}.
$$Then by using  successively  Lemma \ref{commu} and
 the continuity of the Riesz transform (i.e. Proposition \ref{properties}-(3)),  we get
\begin{eqnarray*}
\Big\|\big[\psi_q(\hbox{D}), S_{q-1}(\partial_3\Delta^{-1}(\omega_{\theta}/r))\big]\Delta_q(x_1\rho)\Big\|_{L^{p}}&\lesssim& \|x\psi_q\|_{L^1}\|S_{q-1}(\nabla\partial_3\Delta^{-1}(\omega_{\theta}/r))\|_{L^p}\|\Delta_q(x_1\rho)\|_{L^\infty}\\&\lesssim&
2^{-q}\|x\psi\|_{L^1}\|\nabla\partial_3\Delta^{-1}(\omega_{\theta}/r))\|_{L^p}\|\Delta_q(x_1\rho)\|_{L^\infty}\\
&\lesssim&2^{-q}\|\omega_{\theta}/r\|_{L^p}\|\Delta_q(x_1\rho)\|_{L^\infty}.
\end{eqnarray*}
It follows that
\begin{equation}
\nonumber \|\hbox{III}_1\|_{L^p}
\lesssim \sum_{q\geq 0} \|\omega_{\theta}/r\|_{L^p}\|\Delta_q(x_1\rho)\|_{L^\infty} \lesssim
 \|\omega_{\theta}/r\|_{L^p}\|x_1\rho\|_{B_{\infty,1}^0}
\end{equation}
This proves that the linear operator $T$
$$f\mapsto \sum_{q\geq 0}\partial_1\big\{[\psi_q(\hbox{D}), S_{q-1}(\partial_3\Delta^{-1}f)]\Delta_q(x_1\rho)\big\}$$
is continuous from $L^p$ into itself for every $p\in]1,\infty[$ and that 
$$
\|T\|_{ \mathcal{L}(L^p)}\le C_{p}\|x_1\rho\|_{B_{\infty,1}^0}.
$$
Consequently, by using the interpolation result of  Theorem \ref{interpol}, we get that  $T$   is continuous on  $L^{p,q}$ for every $1<p<\infty$ and $q\in[1,\infty].$ In particular, this yields 
\begin{eqnarray}\label{com002}
\|\hbox{III}_1\|_{L^{3,1}}
&\lesssim&\|\omega_{\theta}/r\|_{L^{3,1}}\|x_1\rho\|_{B_{\infty,1}^0}
\end{eqnarray}
For the term $\hbox{III}_3$,  we use split it into  
\begin{eqnarray*}
\hbox{III}_3&=&\partial_1\sum_{q\geq1}[\rz, \Delta_{q}(\partial_3\Delta^{-1}(\omega_{\theta}/r))]\tilde\Delta_q(x_1\rho)
+\partial_1\sum_{-1\le q\leq0}[\rz, \Delta_{q}(\partial_3\Delta^{-1}(\omega_{\theta}/r))]\tilde\Delta_q(x_1\rho)\\
&:=&\hbox{III}_{31}+\hbox{III}_{32}.
\end{eqnarray*}
Let $p\in]1,\infty[$  from the  Bernstein inequality, we have that 
\begin{eqnarray*}
\|\Delta_k\hbox{III}_{31}\|_{L^p}&\lesssim& 2^{k}\sum_{q\geq k-4}\big\|[\rz, \Delta_{q}(\partial_3\Delta^{-1}(\omega_{\theta}/r))]\tilde\Delta_q(x_1\rho)\big\|_{L^p}
\end{eqnarray*}
 and  the terms inside the sum  can be controlled without using  the structure of the commutator.
  We just  write
\begin{eqnarray*}
\big\|[\rz, \Delta_{q}(\partial_3\Delta^{-1}(\omega_{\theta}/r))]\tilde\Delta_q(x_1\rho)\big\|_{L^p}
&\lesssim& \big\|\Delta_{q}(\partial_3\Delta^{-1}(\omega_{\theta}/r))\|_{L^p}\|\tilde\Delta_q(x_1\rho)\big\|_{L^\infty}\\
&+& \|\Delta_{q}(\partial_3\Delta^{-1}(\omega_{\theta}/r))\|_{L^p}\|\rz\tilde\Delta_q(x_1\rho)\|_{L^\infty}\\
&\lesssim&2^{-q} \|\omega_{\theta}/r\|_{L^p}\|\tilde\Delta_q(x_1\rho)\|_{L^\infty}.
\end{eqnarray*}
Note that we have used the Bernstein inequality, the continuity of the Riesz transform on $L^p$
 and the fact that the support of the Fourier transform of $\tilde{\Delta}_{q}$ does not contains zero
   which gives that the operator $\mathcal{R}_{ij}\Delta_{q}$ also acts continuouly on $L^\infty$
    (since it  can be written as the convolution with an $L^1$ fonction).
It follows that for every $p \in ]1, + \infty[$, we have
\begin{eqnarray*}
\|\hbox{III}_{31}\|_{L^p}&\lesssim&\|\omega_{\theta}/r\|_{L^p} \sum_{k\geq -1}\sum_{q\geq k-4}2^{k-q}\|\tilde\Delta_q(x_1\rho)\|_{L^\infty}
\lesssim \|\omega_{\theta}/r\|_{L^p}\|x_1\rho\|_{B_{\infty,1}^0}.
\end{eqnarray*}
By using again the interpolation result of  Theorem \ref{interpol}, this yields
\begin{eqnarray*}
\|\hbox{III}_{31}\|_{L^{3,1}}&\lesssim&\|\omega_{\theta}/r\|_{L^{3,1}}\|x_1\rho\|_{B_{\infty,1}^0}.
\end{eqnarray*}
We can also  estimate  the term $\hbox{III}_{32}$  without using the structure of the commutator.
 By using  the continuity of the Riesz transform on $L^2$ and \eqref{ndelta-1}, we obtain
\begin{eqnarray*}
\big\|[\rz, \Delta_{q}(\partial_3\Delta^{-1}(\omega_{\theta}/r))]\tilde\Delta_q(x_1\rho)\big\|_{L^2}
&\lesssim& \big\|\Delta_{q}(\partial_3\Delta^{-1}(\omega_{\theta}/r))\|_{L^\infty}\|\tilde\Delta_q(x_1\rho)\big\|_{L^2}\\
&+& \|\Delta_{q}(\partial_3\Delta^{-1}(\omega_{\theta}/r))\|_{L^\infty}\|\rz\tilde\Delta_q(x_1\rho)\|_{L^2}\\
&\lesssim& \|\partial_3\Delta^{-1}(\omega_{\theta}/r)\|_{L^\infty}\|\tilde\Delta_q(x_1\rho)\|_{L^2}\\
&\lesssim& \|\omega_{\theta}/r\|_{L^{3,1}}\|x_1\rho\|_{L^2}.
\end{eqnarray*}
Therefore we get
\begin{eqnarray*}
\|\hbox{III}_{32}\|_{B_{2,1}^{\frac12}}&\lesssim&\|\omega_{\theta}/r\|_{L^{3,1}}\|x_1\rho\|_{L^2}.
\end{eqnarray*}
Consequently we obtain
\begin{equation}\label{com001}
\|\hbox{III}_3\|_{B_{3,1}^0}\lesssim  \|\omega_{\theta}/r\|_{L^{3,1}}\|x_1\rho\|_{B_{\infty,1}^0\cap L^2}.
\end{equation}
Let us now turn  to the estimate of the term  $\hbox{III}_2.$ We write
\begin{eqnarray*}
\hbox{III}_2&=&\sum_{q\geq 0}  \big[ \mathcal{R}_{ij},  \Delta_{q}(\partial_{13}\Delta^{-1}(\omega_{\theta}/r) \big]
 S_{q-1}(x_{1} \rho)+ \big[ \mathcal{R}_{ij},   \Delta_{q}(\partial_{3}\Delta^{-1}(\omega_{\theta}/r) \big]
  \partial_{1} S_{q-1}(x_{1} \rho) \\
&=&\hbox{III}_{21}+\hbox{III}_{22}.
\end{eqnarray*}
We have by definition of the paraproducts that 
\begin{eqnarray*}
\hbox{III}_{21}=\rz(T_{x_1\rho}\mathcal{R}_{13}(\omega_{\theta}/r))-T_{\rz(x_1\rho)}\mathcal{R}_{13}(\omega_{\theta}/r).
\end{eqnarray*}
Thanks to  Proposition \ref{properties},  we get that 
\begin{eqnarray*}
\|\hbox{III}_{21}\|_{L^{3,1}}&\lesssim& \|\mathcal{R}_{13}(\omega_{\theta}/r)\|_{L^{3,1}}\big(\|x_1\rho\|_{L^\infty}+\|\mathcal{R}_{13}(x_1\rho)\|_{L^\infty}\big)\\
&\lesssim& \|\omega_{\theta}/r\|_{L^{3,1}}\big(\|x_1\rho\|_{L^\infty}+\|\mathcal{R}_{13}(x_1\rho)\|_{L^\infty}\big)\\
&\lesssim& \|\omega_{\theta}/r\|_{L^{{3,1}}}\|x_1\rho\|_{B_{\infty,1}^0\cap L^2}.
\end{eqnarray*}
Note that the $L^2$ norm in the right hand-side comes from the low frequency term
 in the Littlewood-Paley  decomposition: we have  
\begin{equation}
\label{remlow} \|\mathcal{R}_{13}  \Delta_{-1}\big(x_{1} \rho) \|_{L^\infty} \lesssim   \|\mathcal{R}_{13}  \Delta_{-1}\big(x_{1} \rho) \|_{L^2} \lesssim   \|  x_{1} \rho \|_{L^2}\end{equation}
thanks to the Berntein inequality and the $L^2$ continuity of the Riesz transform.

For the estimate of  $\hbox{III}_{22}$, we shall use that thanks to the Bernstein inequality,
 we have   for every $f$ that, 
$$ \|\Delta_{q} \partial_{3} \Delta^{-1} f \|_{L^p} \lesssim  2^{-q} \|f\|_{L^p}, \quad \forall q \geq 0, 
 \, p \in ]1, + \infty[.$$
This yields
\begin{eqnarray*}
\|\hbox{III}_{22}\|_{L^p}&\lesssim&  \|\omega_{\theta}/r\|_{L^{p}}\sum_{q\geq 0}2^{-q} \big(\|\partial_1S_{q-1}(x_1\rho)\|_{L^\infty}+\|\partial_1 S_{q-1}\rz(x_1\rho)\|_{L^\infty}\big)\\
 &\lesssim&\|\omega_{\theta}/r\|_{L^{p}} \sum_{q \geq 0 }\sum_{q-2\geq p\geq -1}2^{p-q} \big(\|\Delta_p(x_1\rho)\|_{L^\infty}+\|\Delta_p\rz(x_1\rho)\|_{L^\infty}\big)\\
 &\lesssim&\|\omega_{\theta}/r\|_{L^{p}}\big(\|x_1\rho\|_{B_{\infty,1}^0}+\|\rz(x_1\rho)\|_{B_{\infty,1}^0}\big)\\
 &\lesssim&\|\omega_{\theta}/r\|_{L^{p}}\|x_1\rho\|_{B_{\infty,1}^0\cap L^2}
\end{eqnarray*}
by using again \eqref{remlow}. Consequently, 
by interpolation, we also find
$$
\|\hbox{III}_{22}\|_{L^{3,1}}\lesssim\|\omega_{\theta}/r\|_{L^{3,1}}\|x_1\rho\|_{B_{\infty,1}^0\cap L^2}.
$$
We have thus shown that 
\begin{equation}\label{com00}
\|\hbox{III}_2\|_{L^{3,1}}\lesssim \|\omega_{\theta}/r\|_{L^{3,1}}\|x_1\rho\|_{B_{\infty,1}^0\cap L^2}.
\end{equation}
Gathering  \eqref{com002}, \eqref{com001} and \eqref{com00} , we obtain
\begin{equation}\label{com003}
\|\hbox{III}\|_{L^{3,1}}\lesssim \|\omega_{\theta}/r\|_{L^{3,1}}\|x_1\rho\|_{B_{\infty,1}^0\cap L^2}.
\end{equation}
Finally we obtain
$$
\big\|\partial_1\{[\rz, v^1]\rho\}\big\|_{L^{3,1}}\lesssim \|\omega_{\theta}/r\|_{L^{3,1}}\big(\|x_1\rho\|_{B_{\infty,1}^0\cap L^2}+\|\rho\|_{B_{2,1}^{\frac12}}\big)
$$
thanks to \eqref{com003}, \eqref{eq0076}, \eqref{c1} and \eqref{decompo}.
In the same way, we also obtain the  estimate
$$
\big\|\partial_2\{[\rz, v^2]\rho\}\big\|_{L^{3,1}}\lesssim \|\omega_{\theta}/r\|_{L^{3,1}}\big(\|x_2\rho\|_{B_{\infty,1}^0\cap L^2}+\|\rho\|_{B_{2,1}^{\frac12}}\big).
$$
In view of \eqref{com1}, it remains to estimate the term $\partial_3 \big( [\rz, v^3]\rho\big)$  which has a different structure. 
$\bullet$ {\it Estimate of $\partial_3\big( [\rz, v^3]\rho\big).$} Since we can write that  
\begin{eqnarray*}
\Delta v^3= - (\mbox{curl }\omega)_{3}= - \big( \partial_{r} \omega_{\theta}+ {\omega_{\theta_{r}}}\big)
= -  
\big(r\partial_r({\omega_{\theta}\over r})+2{\omega_{\theta}\over r } \big)
= -  x_h\cdot\nabla_h({\omega_{\theta}\over r}) - 2{\omega_{\theta}\over r}, 
\end{eqnarray*}
we obtain that 
\begin{eqnarray*}
 v^3(x)&=& -  \Delta^{-1}\big(x_h\cdot\nabla_h({\omega_{\theta} \over r})\big) - 2\Delta^{-1}({\omega_{\theta} \over r}).
\end{eqnarray*}
and hence by using  Lemma  \ref{ident001}  that 
\begin{eqnarray}\label{identity5}
\nonumber -  v^3(x)&=& x_h\cdot\nabla_h\Delta^{-1}(\omega_{\theta}/r)-2\sum_{i=1}^2\Delta^{-1}\mathcal{R}_{ii} (\omega_{\theta}/r)+2\Delta^{-1}(\omega_{\theta}/r)\\
&=&x_h\cdot\nabla_h\Delta^{-1}(\omega_{\theta}/r)+2\Delta^{-1}\mathcal{R}_{33}(\omega_{\theta}/r).
\end{eqnarray}
Thus, we have a decomposition of  the commutator  under the  form
\begin{eqnarray*}
 - \partial_3\Big([\rz, v^3]\rho\Big)
&= &\sum_{k=1}^2\partial_3\Big(\partial_k\Delta^{-1}(\omega_{\theta}/r)\big[\rz,x_k\big]\rho\Big)
+2\partial_3\Big(\big[\rz, \Delta^{-1}\mathcal{R}_{33}(\omega_{\theta}/r)\big]\rho\Big)\\
&& + \sum_{k=1}^2\partial_3\Big(\big[\rz, \partial_k\Delta^{-1}(\omega_{\theta}/r)\big](x_k\rho)\Big)\\
&=&\overline{\hbox{I}}+\overline{\hbox{II}}+\overline{\hbox{III}} .
\end{eqnarray*}
To estimate the first  term  $\overline{\hbox{I}}$, we use  Lemma \ref{ident001}-(2)  to obtain that 
\begin{eqnarray*}
\partial_3\Big(\partial_k\Delta^{-1}( {\omega_{\theta} \over r})\big[\rz,x_k\big]\rho\Big)&=&\partial_3\Big(\partial_k\Delta^{-1}({\omega_{\theta} \over r})\mathcal{L}_{ij}^k\rho\Big)
=\mathcal{R}_{3k}({\omega_{\theta} \over r})\,\mathcal{L}_{ij}^k\rho+\partial_k\Delta^{-1}( {\omega_{\theta}
 \over r})\partial_3\mathcal{L}_{ij}^k\rho.
\end{eqnarray*}
 It follows that 
\begin{eqnarray*}
\|\overline{\hbox{I}}\|_{L^{3,1}}&\le&\sum_{k=1}^2\Big(\|\mathcal{L}_{ij}^k\rho\|_{L^\infty}\|\mathcal{R}_{3k}(\omega_{\theta}/r)\|_{L^{3,1}}+\|\partial_k\Delta^{-1}(\omega_{\theta}/r)\|_{L^\infty}\|\partial_3\mathcal{L}_{ij}^k\rho\|_{L^{3,1}}\Big)\\
&\lesssim&\|\rho\|_{L^{3,1}}\|\omega_{\theta}/r\|_{L^{3,1}}
\end{eqnarray*}
 thanks to \eqref{Lijkinfty}, \eqref{Lijk3}.
The estimates of the  terms $\overline{\hbox{II}}$ and $\overline{\hbox{III}}$   are similar to the ones  
 of   ${\hbox{II}}$ and $\hbox{III}$ in   \eqref{decompo}
  (indeed,  the operator $ \Delta^{-1}\mathcal{R}_{33}= \partial_{33}\Delta^{-2}$ has the same properties
   as $\mathcal{L}= -2 \partial_{13} \Delta^{-2}$ which arises in \eqref{decompo}) consequently, 
    we also get  
   as  in   \eqref{eq0076} and \eqref{com003} that 
$$
\|\overline{\hbox{I}}\|_{L^{3,1}}\le\|\omega_{\theta}/r\|_{L^3}\|\rho\|_{B_{2,1}^{\frac12}}, \quad 
\|\overline{\hbox{II}}\|_{L^{3,1}}\lesssim\|\omega_{\theta}/r\|_{L^{3,1}}\|\rho x_h\|_{B_{\infty,1}^0\cap L^2}.
$$
Consequently, we also find that 
$$
\big\|\partial_3\big([\rz, v^3]\rho \big)\big\|_{L^{3,1}}\lesssim \|\omega_{\theta}/r\|_{L^{3,1}}\big(\|\rho x_h\|_{B_{\infty,1}^0\cap L^2}+\|\rho\|_{B_{2,1}^{\frac12}}\big).
$$
This ends the proof of  Theorem  \ref{prop2}.
\end{proof}
\subsection{Commutation between the advection operator  and  $\Delta_q$}
\label{subsec2}
 The last commutator estimate which is needed in the proof of our main result is the following.
\begin{prop}\label{propcomm}
Let  $v$ be an axisymmetric  divergence free vector field  without swirl   and $\rho$ a smooth scalar function. Then   there exists $C>0$ such that for  every $q\in\NN\cup\{-1\}$ we have
$$
\big\|\big[\Delta_q,v\cdot\nabla\big]\rho\big\|_{L^{2}}\le C\|\omega_{\theta}/r\|_{L^{3,1}}\big(\|\rho\, x_h\|_{L^{6}}+\|\rho\|_{L^2}\big).
$$

\end{prop}
\begin{proof}
From the incompressibility of the velocity we have
\begin{eqnarray}
\label{equa0}
\big[\Delta_q,v\cdot\nabla\big]\rho&=&\sum_{i=1}^3\partial_i\Big(\big[\Delta_q,v^i\big]\rho\Big)
=\hbox{I}+\hbox{II}+\hbox{III}.
\end{eqnarray}
The first and the second terms  can  be  handled in  the same way, so we shall only detail the proof of the  estimate  of the first one.  Thanks  to \eqref{ident1},  we have that 
\begin{equation*}
v^1(x)=x_1\Delta^{-1}\partial_3(\omega_{\theta}/r)+\mathcal{L}(\omega_{\theta}/r),\quad\hbox{with}\quad \mathcal{L}=-2\mathcal{R}_{13}\Delta^{-1}
\end{equation*}
and hence we get
\begin{eqnarray*}
\hbox{I}&=&\partial_1\Big(\Big[\Delta_q,x_1\Delta^{-1}\partial_3(\omega_{\theta}/r)\Big]\rho\Big)+\partial_1\Big(\Big[\Delta_q,\mathcal{L}(\omega_{\theta}/r)\Big]\rho\Big)\\
&=&\hbox{I}_1+\hbox{I}_2.
\end{eqnarray*}
The  estimate of  the second term in  the right-hand side  is again a direct consequence of  Lemma \ref{lemcom}
 and \eqref{Lijinfty}. Indeed, we write 
 \begin{eqnarray*}
\|\hbox{I}_2\|_{L^2}&\lesssim \big\|\nabla \mathcal{L}(\omega_{\theta}/r)\big\|_{L^\infty}\|\rho\|_{L^2}
\lesssim \big\|\omega_{\theta}/r\big\|_{L^{3,1}}\|\rho\|_{L^2}.
\end{eqnarray*}
The first term $I_{1}$ in  the right-hand side can be expanded  under the form
\begin{eqnarray*}
\hbox{I}_1&=&\partial_1\Big(\Big[\Delta_q,\Delta^{-1}\partial_3(\omega_{\theta}/r)\Big](x_1\rho)\Big)+\partial_1\Big(\Delta^{-1}\partial_3(\omega_{\theta}/r)\big[\Delta_q, x_1\big]\rho\Big)
= \hbox{I}_{11}+\hbox{I}_{12}.
\end{eqnarray*}
We start with the estimate of $\hbox{I}_{12}.$  By definition of $\Delta_{q}$, we have
\begin{eqnarray*}
x_1\Delta_q\rho&=&x_12^{3q}\int_{\RR^3}\varphi(2^q(x-y))\rho(y) dy\\
&=&2^{3q}\int_{\RR^3}\varphi(2^q(x-y))y_1\rho(y) dy+2^{3q}\int_{\RR^3}\varphi(2^q(x-y))(x_1-y_1)\rho(y) dy\\
&=&\Delta_q(x_1\rho)+2^{-q}2^{3q}\,\varphi_1(2^{q}\cdot)\star\rho,
\end{eqnarray*}
where $\varphi_1(x)=x_1\varphi(x)\in \mathcal{S}(\RR^3).$ Consequently we get
 the expression of the commutator: 
\begin{equation}\label{conv12}
\big[\Delta_q, x_1\big]\rho=-2^{-q}2^{3q}\,\varphi_1(2^{q}\cdot)\star\rho.
\end{equation}
This yields
\begin{eqnarray*}
\hbox{I}_{12}&=&-\big(\mathcal{R}_{13}(\omega_{\theta}/r)\big) 2^{2q}\,\varphi_1(2^{q}\cdot)\star\rho-\big\{\Delta^{-1}\partial_3(\omega_{\theta}/r)\big\}2^{3q}\,(\partial_1\varphi_1)(2^{q}\cdot)\star\rho.
\end{eqnarray*}
Therefore we get by using again the  H\"{o}lder inequality, the continuity of the Riesz transform, 
 \eqref{ndelta-1} and the  Young inequality for convolutions that: 
 \begin{eqnarray*}
\big\|\hbox{I}_{12}\|_{L^2}&\le&\big\|\mathcal{R}_{13}(\omega_{\theta}/r)\big\|_{L^3}2^{2q}\big\|\varphi_1(2^{q}\cdot)\star\rho\big\|_{L^{6}}
+\big\|\Delta^{-1}\partial_3(\omega_{\theta}/r)\big\|_{L^\infty}2^{3q}\big\|(\partial_1\varphi_1)(2^{q}\cdot)\star\rho\big\|_{L^2}\\
&\lesssim&\|\omega_{\theta}/r\|_{L^3}\|\varphi_1\|_{L^{\frac32}}\|\rho\|_{L^2}+\|\omega_{\theta}/r\|_{L^{3,1}}\|\partial_1\varphi_1\|_{L^1}\|\rho\|_{L^2}\\
&\lesssim&\|\omega_{\theta}/r\|_{L^{3,1}}\|\rho\|_{L^2}.
\end{eqnarray*}
Note that we have also used the embedding \eqref{propLp1}.

To estimate $\hbox{I}_{11}$ we use again  Lemma \ref{lemcom}: 
\begin{eqnarray*}
\|\hbox{I}_{11}\|_{L^2}&\lesssim&\|\nabla\Delta^{-1}\partial_3(\omega_{\theta}/r)\|_{L^3}\|x_1\rho\|_{L^{6}}
\lesssim\|\omega_{\theta}/r\|_{L^3}\|x_1\rho\|_{L^{6}}
\lesssim \|\omega_{\theta}/r\|_{L^{3,1}}\|x_1\rho\|_{L^{6}}.
 \end{eqnarray*}

We have thus shown that 
\begin{equation}\label{equa1}
\|\hbox{I}\|_{L^2}\lesssim \|\omega_{\theta}/r\|_{L^{3,1}}\big(\|\rho\|_{L^p}+\|x_1\rho\|_{L^{6}}\big).
\end{equation}
In the same way, we obtain that 
\begin{equation}\label{equa2}
\|\hbox{II}\|_{L^2}\lesssim \|\omega_{\theta}/r\|_{L^{3,1}}\big(\|\rho\|_{L^2}+\|x_2\rho\|_{L^{6}}\big).
\end{equation}
It remains to estimate the last term $\hbox{III}.$  By using  \eqref{identity5},  we get
\begin{eqnarray*}
- \hbox{III}&=&\partial_3\Big\{\Big[\Delta_q,\nabla_h\Delta^{-1}(\omega_{\theta}/r)\Big](x_h\rho)\Big\}+\partial_3\Big\{\nabla_h\Delta^{-1}(\omega_{\theta}/r)\big[\Delta_q,x_h\big]\rho\Big\} \\ 
 & + &  2\partial_3\Big\{\Big[\Delta_q,\Delta^{-1}\mathcal{R}_{33}(\omega_{\theta}/r)\Big]\rho\Big\}\\
 &=&\hbox{III}_1+\hbox{III}_2+\hbox{III}_3.
\end{eqnarray*}
The estimates of the  first  and last terms  follow again  from  Lemma \ref{lemcom}: we write that
\begin{eqnarray*}
\|\hbox{III}_1\|_{L^2}&\le& C\|\nabla^2\Delta^{-1}(\omega_{\theta}/r)\|_{L^3}\|x_h\rho\|_{L^{6}}
\lesssim  \|\omega_{\theta}/r\|_{L^{3}}\|x_h\rho\|_{L^{6}}
\lesssim \|\omega_{\theta}/r\|_{L^{3,1}}\|x_h\rho\|_{L^{6}}
\end{eqnarray*}
and that 
\begin{eqnarray*}
\|\hbox{III}_3\|_{L^2}&\le& C\|\nabla\Delta^{-1}\mathcal{R}_{33}(\omega_{\theta}/r)\|_{L^\infty}\|\rho\|_{L^2}
\lesssim \|\mathcal{R}_{33}(\omega_{\theta}/r)\|_{L^{3,1}}\|\rho\|_{L^2}
\lesssim C\|\omega_{\theta}/r\|_{L^{3,1}}\|\rho\|_{L^2}.
\end{eqnarray*}
Note that we have used again the estimate \eqref{ndelta-1}.

Finally, 
to estimate the second term $\hbox{III}_2$ we  can use the expression of the commutator
 $[\Delta_{q}, x_{h}]$ given by  \eqref{conv12}: 
\begin{eqnarray*}
\hbox{III}_2&=&2^{-q}\partial_3\Big(\big(\nabla_h\Delta^{-1}(\omega_{\theta}/r)\big) \, 2^{3q}\varphi_h(2^{q}\cdot)\star \rho \Big)\\
&=&2^{-q}\big(\partial_3\nabla_h\Delta^{-1}(\omega_{\theta}/r)\big) \,\big( 2^{3q}\varphi_h(2^{q}\cdot)\star \rho \big)+\nabla_h\Delta^{-1}(\omega_{\theta}/r) \, \big(2^{3q}(\partial_3\varphi_h)(2^{q}\cdot)\star \rho \big),
\end{eqnarray*}
with $\varphi_h(x)=-x_h\varphi(x).$
It follows as before  that
\begin{eqnarray*}
\|\hbox{III}_2\|_{L^2}&\lesssim&2^{-q}\|\omega_{\theta}/r\|_{L^3} 2^{3q}\|\varphi_h(2^{q}\cdot)\star \rho\|_{L^{6}}
+\|\nabla_h\Delta^{-1}(\omega_{\theta}/r)\|_{L^\infty}  2^{3q}\|(\partial_3\varphi_h)(2^{q}\cdot)\star \rho \|_{L^2}\\
&\lesssim&\|\omega_{\theta}/r\|_{L^3} \|\varphi_h\|_{L^{\frac32}}\| \rho\|_{L^2}
+\|\omega_{\theta}/r\|_{L^{3,1}}  \|\partial_3\varphi_h\|_{L^1}\| \rho \|_{L^2}\\
&\lesssim&\|\omega_{\theta}/r\|_{L^{3,1}} \| \rho \|_{L^2}.
\end{eqnarray*}
Gathering  these estimates we also  find that 
$$
\|\hbox{III}\|_{L^3}\lesssim \|\omega_{\theta}/r\|_{L^{3,1}}\big( \| \rho \|_{L^p}+\|x_h\rho\|_{L^{6}}\big).
$$
In view of \eqref{equa0}, \eqref{equa1}, \eqref{equa2} and the last estimate, this ends
 the proof of Proposition \ref{propcomm}.
\end{proof}

\section{A priori estimates}
\label{sectionapriori}
In this section  we intend to  establish the global  a priori estimates needed for the proof of \mbox{Theorem \ref{thm1}.}   We shall  first  prove  some basic  weak  estimates that can be  obtained easily  through energy
 type  estimates.   In  a second step, we shall prove  
 the control of some stronger norms such as $\|\omega(t)\|_{L^\infty}$ and $\|\nabla v(t)\|_{L^\infty}$. This part requires more refined analysis: we use the  special structure of the Boussinesq  model combined with the  previous commutator estimates. 
\subsection{Energy estimates}
We  start with some elementary  energy estimates.
\begin{prop}\label{Energy}
Let $(v,\rho)$  be a smooth solution of \eqref{bsintro} then
\begin{enumerate}
\item For $p\in]1,\infty[,q\in[1,\infty]$ and $t\in\RR_+$,  we have
$$
\|\rho\|_{L^\infty_tL^2}^2+2\|\nabla\rho\|_{L^2_tL^2}^2\le\|\rho_0\|_{L^2}^2,\quad\|\rho\|_{L^{\infty}_tL^{p,q}}\le C\|\rho_0\|_{L^{p,q}}. 
$$
\item For $v_0\in L^2, \rho_0\in L^2$ and $t\in\RR_+$ we have
$$
\|v(t)\|_{L^2}\le\|v_0\|_{L^2}+t\|\rho_0\|_{L^2}.
$$

\item For $\rho_0\in  L^2$ we have the dispersive estimate
$$
\|\rho(t)\|_{L^\infty}\le C\frac{\|\rho_0\|_{L^2}}{t^{\frac34}}\cdot
$$
The constant $C$ is absolute.
\end{enumerate}

\end{prop}
Note that  the axisymmetric assumption is not needed in this proposition
\begin{proof}
{\bf$(1)$} By  taking the $L^2$-scalar product  of the  second equation of  \eqref{bsintro}  with  $\rho$ and integrating by parts, we get since $v$ is divergence free that
$$
\frac{1}{2}\frac{d}{dt}\|\rho(t)\|_{L^2}^2+\int_{\RR^3}|\nabla\rho(t,x)|^2\, dx=0.
$$
Integrating in time  this differential inequality gives the desired result.

Let us now move to the estimate of the density in Lorentz spaces. First, 
 the same argument yields that for every $p\in[1,\infty]$,  we have
$$
\|\rho(t)\|_{L^p}\le\|\rho_0\|_{L^p}.
$$
It suffices now to use the interpolation result  of Theorem \ref{interpol}. 

{\bf{$(2)$}}
We take the $L^2$-scalar  product of  the velocity equation with $v$ and we integrate by parts 
\begin{equation*}
\frac12\frac{d}{dt}\|v(t)\|_{L^2}^2\le\|v(t)\|_{L^2}\|\rho(t)\|_{L^2}
\end{equation*}
 and this implies  that
$$
\frac{d}{dt}\|v(t)\|_{L^2}\le\|\rho(t)\|_{L^2}.
$$
 Thus, integrating  in time gives 
$$
\|v(t)\|_{L^2}\le\|v_0\|_{L^2}+\int_0^t\|\rho(\tau)\|_{L^2}d\tau.
$$
Since  $\|\rho(t)\|_{L^2} \leq \|\rho_0\|_{L^2},$  we infer
$$
\|v(t)\|_{L^2}\le\|v_0\|_{L^2}+t\|\rho_0\|_{L^2}.
$$

{\bf$(3)$} The estimate is a direct consequence of Lemma \ref{Nashlem}
The proof of  the proposition is now achieved.
\end{proof}

\subsection{Estimates of the moments of $\rho.$}
We have seen in subsection \ref{subsec1} and subsection \ref{subsec2} that the estimates of the commutators involve some moments  of the density. Thus we aim in this paragraph  at giving suitable estimates for the moments that will be needed  later when  we shall perform our   diagonalization of the Boussinesq system.
 Two types of estimates are    discussed:  the energy estimates of the horizontal moments $|x_h|^k\rho$, with $k=1,2$ and some dispersive estimates. More precisely we prove the following.

\begin{prop}\label{prop07}
Let $v$ be a vector field  with zero divergence and satisfying the energy estimate of Proposition \ref{Energy}. Let $\rho$ be a solution of the transport-diffusion equation
$$
\partial_t\rho+v\cdot\nabla\rho-\Delta\rho=0,\quad\rho(0,x)=\rho_0.
$$
Then we have the following estimates.
\begin{enumerate}
\item For $\rho_0\in L^2$ and $x_h\rho_0\in L^2$ , there exists $C_0>0$ such that for every $t\geq0$
$$
\|x_h\rho\|_{L^\infty_tL^2}+\|x_h\rho\|_{L^2_t\dot{H}^1}\lesssim C_0\big(1+t^{\frac54}\big).
$$
\item For  $\rho_0\in L^2\cap L^m, m>6$ and $x_h\rho_0\in L^2$ , there exists $C_0>0$ such that for every $t>0$
\begin{equation*}\label{mom1}
\|x_h\rho(t)\|_{L^\infty}\le C_0(t^{\frac14}+t^{-\frac34}).
\end{equation*}
\item For   $\rho_0\in L^2$ and $|x_h|^2\rho_0\in L^2$, there exists $C_0>0$ such that for every $t\geq0$
$$
\||x_h|^2\rho\|_{L^\infty_tL^2}+\||x_h|^2\rho\|_{L^2_t\dot{H}^1}\le C_0\big(1+t^{\frac52}\big).
$$
\item For   $ \rho_0\in L^2\cap L^6$ and  $|x_h|^2\rho_0\in L^2$, there exists $C_0>0$ such that for every $t>0$
\begin{equation*}\label{mom2}
\||x_h|^2\rho(t)\|_{L^6}\le C_0(t^{\frac{13}{6}}+t^{-\frac12}).
\end{equation*}

\end{enumerate}
\end{prop}
\begin{Rema}
 Note that when $\rho_0\in L^2$ and $|x_h|^2\rho_0\in L^2$ then automatically the moment of order one belongs to $L^2$, that is $x_h\rho_0\in L^2.$ This is an easy  consequence of the  H\"{o}lder inequality
$$
\|x_h\rho\|_{L^2}\le \|\rho\|_{L^2}^{\frac12}\||x_h|^2\rho\|_{L^2}^{\frac12}.
$$
\end{Rema}
\begin{proof}
{\bf{$(1)$}}
Setting $f=x_h\rho$, we can easily check that  $f$ solves the equation
\begin{equation}\label{disper1}
\partial_t f+v\cdot\nabla f-\Delta f=v^h\rho-2\nabla_h\rho
\end{equation}
with the notations $v^h=(v^1,v^2)$ and $\nabla_h=(\partial_1,\partial_2).$
Now, taking the $L^2$-scalar product \mbox{with $f$, }integrating by parts and using the  H\"{o}lder inequality
\begin{eqnarray*}
\frac{1}{2}\frac{d}{dt}\|f(t)\|_{L^2}^2+\|\nabla f(t)\|_{L^2}^2&=&\int_{\RR^3}v^h\rho fdx-2\int_{\RR^3}\nabla_h\rho\, fdx\\
&\le&\| v\|_{L^2}\|\rho\|_{L^{3}}\|f\|_{L^6}+2\|\rho\|_{L^2}\|\nabla f\|_{L^2}.
\end{eqnarray*}
 By using  the Sobolev embedding $\dot{H}^1\hookrightarrow L^6$ combined with the  Young inequality, we obtain
\begin{eqnarray*}
\frac{d}{dt}\|f(t)\|_{L^2}^2+\|\nabla f(t)\|_{L^2}^2\lesssim \| v\|_{L^2}^2\|\rho\|_{L^{3}}^2+\|\rho\|_{L^2}^2.
\end{eqnarray*}
  Since the Gagliardo-Nirenberg inequality gives that 
 $$
\|\rho\|_{L^3}^2\le\|\rho\|_{L^2}\|\nabla\rho\|_{L^2}, 
$$
 we infer
\begin{eqnarray}\label{eqrs1}
\frac{d}{dt}\|f(t)\|_{L^2}^2+\|\nabla f(t)\|_{L^2}^2\lesssim \| v\|_{L^2}^2\|\rho\|_{L^{2}}\|\nabla\rho\|_{L^2}+\|\rho\|_{L^2}^2.
\end{eqnarray}
Integrating in time and using the  energy estimate  of  Proposition \ref{Energy}-(1),  we thus obtain
\begin{eqnarray*}
\|f(t)\|_{L^2}^2+\|\nabla f(t)\|_{L^2_tL^2}^2&\lesssim &\|f_0\|_{L^2}^2+\| v\|_{L^\infty_tL^2}^2\|\rho_0\|_{L^{2}}\|\nabla\rho\|_{L^1_tL^2}+\|\rho_0\|_{L^2}^2 t\\
&\lesssim&\|f_0\|_{L^2}^2+C_0(1+t^2) t^{\frac12}+\|\rho_0\|_{L^2}^2 t\\
&\le& C_0(1+t^{\frac52}).
\end{eqnarray*}

{\bf{$(2)$}}
  We shall  apply  Lemma \ref{Nashlem} to \eqref{disper1} with $F = \rho \, e_{i}$
  and $G= v_{i} \rho.$ First, we observe that we have obviously from the  H\"{o}lder inequality combined with Proposition \ref{Energy}-(1) that for $m\geq2$
  \begin{eqnarray*} \|G\|_{L^\infty_{t }L^{\frac{2m}{m+2}}} &\leq&  \|v \|_{L^\infty_{t}L^2} \|\rho\|_{L^\infty_tL^m}\\
  &\le&C_0(1+t)
  \end{eqnarray*}
  and $$ 
  \| F \|_{L^\infty_{t}L^6} \leq \|\rho_{0}\|_{L^6}.$$
   Consequently, we get from  Lemma \ref{Nashlem} and Proposition \ref{Energy}  that for $m>6$ and for $t>0,$
   \begin{eqnarray*} 
   \nonumber\| f(t) \|_{L^\infty} &\leq&  C \big(1 + {  t^{-{3 \over 4}} } \big)  \|f_{0}\|_{L^2} + 
     C_0(1+ t^{{1 \over 4}-{3\over 2m}})  + (1+t^{\frac14}) \|\rho_{0}\|_{L^6}\\
     &\le& C_{0}(  {t^{-\frac34}} + t^{1 \over 4}).
     \end{eqnarray*}

{\bf{$(3)$}} The second moment $g=|x_h|^2\rho$  solves  the following equation
\begin{eqnarray*}
\partial_t g+v\cdot\nabla g-\Delta g&=&2v^h(x_h\rho)-2\nabla_h\rho-4\textnormal{div}_h(x_h\rho)\\
&=&v_h f-2\nabla_h\rho-4\textnormal{div}_h f.
\end{eqnarray*}
 By using again an $L^2$ energy estimate, we  find that 
\begin{eqnarray*}
\frac{d}{dt}\|g(t)\|_{L^2}^2+\|\nabla g(t)\|_{L^2}^2\lesssim \| v(t)\|_{L^2}^2\|f(t)\|_{L^{2}}\|\nabla f(t)\|_{L^2}+\|\rho(t)\|_{L^2}^2+\|f(t)\|_{L^2}^2.
\end{eqnarray*}
Thus, by  integrating in time and using the energy estimates for $\rho,v$ and $f$ we get 
\begin{eqnarray*}
\|g(t)\|_{L^2}+\|\nabla g \|_{L^2_tL^2}&\lesssim&\|g_0\|_{L^2}+ \| v\|_{L^\infty_tL^2}\|f\|_{L^\infty_tL^{2}}^{\frac12}\|\nabla f\|_{L^2_tL^2}^{\frac12}t^{\frac14}+\|\rho_0\|_{L^2} t^{\frac12}+\|f\|_{L^\infty_tL^2}t^{\frac12}\\
&\le& C_0\big(1+t^{\frac52}\big),
\end{eqnarray*}
where $C_0$ is a constant depending on the quantities $\||x_h|^k\rho_0\|_{L^2}$ for $k=0,1,2.$

{\bf{$(4)$}} By  setting $g_1(t,x)=tg(t,x),$ we have that 
\begin{eqnarray*}
\partial_t g_1+v\cdot\nabla g_1-\Delta g_1&=&g+2v^h(tx_h\rho)-2t\nabla_h\rho-4\textnormal{div}_h(t\,x_h\rho).
\end{eqnarray*}
Multiplying this equation by $|g_1|^4 g_1$, integrating by parts and  the obvious inequality $|x_h\rho|\le |\rho|^{\frac12}|g_1|^{\frac12}$, we thus get 
\begin{eqnarray*}
\frac16\frac{d}{dt} \|g_1\|_{L^6}^6+5\int_{\RR^3}|\nabla g_1|^2|g_1|^4dx&\lesssim&\|g\|_{L^6}\|g_1\|_{L^6}^5+{t}^{\frac12}\int_{\RR^3}|v||\rho|^{\frac12} |g_1|^{\frac{11}{2}}dx\\
&+&t\int_{\RR^3}|\rho||\nabla g_1||g_1|^4dx+t^{\frac12}\int_{\RR^3}|\rho|^{\frac12}|g_1|^{\frac92}|\nabla g_1|dx.
\end{eqnarray*}
It follows from H\"{o}lder inequality that
$$
t^{\frac12}\int_{\RR^3}|\rho|^{\frac12}|g_1|^{\frac92}|\nabla g_1|dx\le {t}^{\frac12}\|v\|_{L^2}\|g_1\|_{L^{18}}^{\frac{11}{2}}\|\rho\|_{L^{\frac{18}{7}}}^{\frac12}
$$
Consequently 
\begin{eqnarray*}
\frac16\frac{d}{dt} \|g_1\|_{L^6}^6+5\int_{\RR^3}|\nabla g_1|^2|g_1|^4dx
&\lesssim&\|g\|_{L^6}\|g_1\|_{L^6}^5+{t}^{\frac12}\|v\|_{L^2}\|g_1\|_{L^{18}}^{\frac{11}{2}}\|\rho\|_{L^{\frac{18}{7}}}^{\frac12}\\
&+&\big(t\|\rho\|_{L^6}\|g_1\|_{L^6}^2+t^{\frac12}\|\rho\|_{L^6}^{\frac12}\|g_1\|_{L^6}^{\frac52}\big)\Big(\int_{\RR^3}|\nabla g_1|^2|g_1|^4dx\Big)^{\frac12}.
\end{eqnarray*}
Now we can use the Young inequality combined with the following  Sobolev inequality 
\begin{eqnarray*}
\|g_1\|_{L^{18}}^{18}\lesssim\|\nabla (g_1^3)\|_{L^2}^2=9\int_{\RR^3}|\nabla g_1|^2|g_1|^4dx
\end{eqnarray*}
 to  obtain that 
\begin{eqnarray*}
\frac{d}{dt} \|g_1\|_{L^6}^6+c\|g_1\|_{L^{18}}^6+c\int_{\RR^3}|\nabla g_1|^2|g_1|^4dx&\lesssim&\|g\|_{L^6}\|g_1\|_{L^6}^5+{t}^6\|v\|_{L^2}^{12}\|\rho\|_{L^{\frac{18}{7}}}^{6}\\
&+&t^{2}\|\rho\|_{L^6}^{2}\|g_1\|_{L^6}^{4}+t\|\rho\|_{L^6}\|g_1\|_{L^6}^{5}.
\end{eqnarray*}
 By using Proposition \ref{Energy}, we deduce that 
\begin{eqnarray*}
\frac{d}{dt} \|g_1\|_{L^6}^6+c\|g_1\|_{L^{18}}^6+c\int_{\RR^3}|\nabla g_1|^2|g_1|^4dx&\lesssim&\|g\|_{L^6}\|g_1\|_{L^6}^5+C_0t^6(1+t^{12})\\
&+&t^{2}\|\rho_0\|_{L^6}^{2}\|g_1\|_{L^6}^{4}+t\|\rho_0\|_{L^6}\|g_1\|_{L^6}^{5}.
\end{eqnarray*}
Next, by using again the  Young inequality, we infer
\begin{eqnarray*}
\frac{d}{dt} \|g_1(t)\|_{L^6}^6&\le&C_0(t^6+t^{18})+C_0\big(t+\|g(t)\|_{L^6}\big)\|g_1(t)\|_{L^6}^5.\\
\end{eqnarray*}
 By integrating in time this differential inequality, we obtain that 
\begin{eqnarray*}
\|g_1(t)\|_{L^6}^6&\le&C_0(t^7+t^{19})+C_0\int_0^t(\tau+\|g(\tau)\|_{L^6})\|g_1(\tau)\|_{L^6}^5d\tau.
\end{eqnarray*}
Therefore we get from Proposition \ref{prop07}-(3) combined with  the Sobolev embedding
$ \dot{H}^1 \subset L^6$ that 
\begin{eqnarray*}
\|g_1(t)\|_{L^6}&\le& C_0(t^{\frac76}+t^{\frac{19}{6}})+C_0\|g\|_{L^1_tL^6}\\
&\le&C_0(t^{\frac76}+t^{\frac{19}{6}})+t^{\frac12}\|\nabla g\|_{L^2_tL^2}\\
&\le&C_0(t^{\frac76}+t^{\frac{19}{6}})+C_0(1+t^{\frac52})t^{\frac12}.
\end{eqnarray*}
 Therefore, we obtain  that
$$
\||x_h|^2\rho(t)\|_{L^6}\le C_0(t^{\frac{13}{6}}+t^{-\frac12}).
$$
This ends  the proof of Proposition \ref{prop07}.
\end{proof}

\subsection{Strong estimates}
 As in the study  of the  axisymmetric Euler equation,  the  main important quantity that one should estimate
  in order to get the global existence of smooth solutions  is $\|\frac{\omega}{r} (t)\|_{L^{3,1}} $.  Indeed, this  will enable  us to bound stronger norms such as  $\|\omega(t)\|_{L^\infty}$ and $\|\nabla v(t)\|_{L^\infty}$ which are the significant quantities to propagate  higher regularities.
\subsubsection{Estimate of $\|\frac{\omega}{r}(t)\|_{L^\infty}$}
 First, we will introduce the following notation: we denote by  $\Phi_k$ any function
of the form 
$$
\Phi_k(t)=  C_{0}\underbrace{ \exp(...\exp  }_{k\,times}(C_0t^{\frac{19}{6}})...),
$$
where $C_{0}$ depends on the involved norms of the initial data and its value may vary from line to line up to some absolute constants. 
We will make an intensive  use (without mentionning it) of  the following trivial facts
$$
\int_0^t\Phi_k(\tau)d\tau\leq \Phi_k(t)\qquad{\rm and}\qquad \exp({\int_0^t\Phi_k(\tau)d\tau})\leq \Phi_{k+1}(t).
$$

We first establish the following result.
\begin{prop}\label{Strong}
Let $v_0\in L^2$ be an axisymmetric vector field such that $\frac{\omega_0}{r}\in L^{3,1}$ and $\rho_0 \in L^2\cap L^{m}, $ for   $m>6$, axisymmetric  and such that $|x_h|^2\rho_0\in L^2$. Then, we have for every $t\in\RR_+$
$$
\big\|\frac{\omega}{r}(t)\big\|_{L^{3,1}}+\big\|\frac{v^r}{r}(t)\big\|_{L^\infty}\le\Phi_2(t),
$$
where $C_0$ is a constant depending on the norms of the initial data.
\end{prop}
\begin{proof}
Recall that the equation of the scalar component of the vorticity $\omega=\omega_\theta e_\theta$ is given by
\begin{equation}
 \label{tourbillon1}
\partial_t \omega_\theta +v\cdot\nabla\omega_\theta=\frac{v^r}{r}\omega_\theta-\partial_r\rho.
\end{equation}
It follows that  the evolution of the quantity $\frac{\omega_\theta}{r}$
is governed by  the equation
\begin{equation}
\label{equation_i}
\big(\partial_t+v\cdot\nabla\big)\frac{\omega_\theta}{r}=-\frac{\partial_r\rho}{r}\cdot
\end{equation}
 By applying the operator $\frac{\partial_r}{r}\Delta^{-1}$ to the equation of the density in \eqref{bsintro}, we
  obtain that 
$$
\big(\partial_t+v\cdot\nabla\big)\big(\frac1r\partial_r\Delta^{-1}\rho\big)-\frac{\partial_r\rho}{r}=-\Big[\frac1r\partial_r\Delta^{-1},v\cdot\nabla  \Big]\rho.
$$ By setting $\Gamma:=\frac{\omega_\theta}{r}+\frac{\partial_r}{r}\Delta^{-1}\rho,$ 
 we infer $$
\big(\partial_t+v\cdot\nabla\big)\Gamma=-\Big[\frac1r\partial_r\Delta^{-1},v\cdot\nabla  \Big]\rho.
$$
Observe that the incompressibility of the velocity field  allows us to get that  for every $p\in[1,\infty]$
$$
\|\Gamma(t)\|_{L^p}\le \|\Gamma_0\|_{L^p}+\int_0^t\Big\|\Big[\frac1r\partial_r\Delta^{-1},v\cdot\nabla  \Big]\rho\Big\|_{L^p}d\tau.
$$
Therefore we get by the interpolation result of Theorem \ref{interpol}   that for $1<p<\infty$ and $q\in[1,\infty]$
$$
\|\Gamma(t)\|_{L^{p,q}}\le\|\Gamma_0\|_{L^{p,q}}+\int_0^t\Big\|\Big[\frac1r\partial_r\Delta^{-1},v\cdot\nabla  \Big]\rho(\tau)\Big\|_{L^{p,q}}d\tau.
$$
In particular, we have
$$
\|\Gamma(t)\|_{L^{3,1}}\le\|\Gamma_0\|_{L^{3,1}}+\int_0^t\Big\|\Big[\frac1r\partial_r\Delta^{-1},v\cdot\nabla  \Big]\rho(\tau)\Big\|_{L^{3,1}}d\tau.
$$
Applying Theorem \ref{prop2} we find
$$
\|\Gamma(t)\|_{L^{3,1}}\le\|\Gamma_0\|_{L^{3,1}}+\int_0^t\|(\omega_{\theta}/r)(\tau)\|_{L^{3,1}}\big(\|x_h\rho(\tau)\|_{B_{\infty,1}^0\cap L^2}+\|\rho(\tau)\|_{B_{2,1}^{\frac12}}\big)d\tau.
$$
 Moreover, thanks to  Proposition \ref{prop1} and Proposition \ref{Energy} we  have 
\begin{eqnarray*}
\|(\omega_{\theta}/r)(t)\|_{L^{3,1}}&\le&\|\Gamma(t)\|_{L^{3,1}}+\big\|\frac1r\partial_r\Delta^{-1}\rho(t)\big\|_{L^{3,1}}
\le\|\Gamma(t)\|_{L^{3,1}}+C\|\rho_0\|_{L^{3,1}}.
\end{eqnarray*}
 The combination of these last estimates yield
$$
\|(\omega_{\theta}/r)(t)\|_{L^{3,1}}\le C \big(\|\omega_0/r\|_{L^{3,1}}+\|\rho_0\|_{L^{3,1}}\big)+\int_0^t\|(\omega_{\theta}/r)(\tau)\|_{L^{3,1}}\big(\|x_h\rho (\tau)\|_{B_{\infty,1}^0\cap L^2}+\|\rho(\tau)\|_{B_{2,1}^{\frac12}}\big)d\tau.
$$ 
Thus we get by the   Gronwall inequality that 
\begin{equation}\label{exp1}
\|(\omega_{\theta}/r)(t)\|_{L^{3,1}}\le C\big(\|\omega_0/r\|_{L^{3,1}}+\|\rho_0\|_{L^{3,1}}\big)  \exp \big( {C\|x_h\rho\|_{L^1_t(B_{\infty,1}^0\cap L^2)}+C\|\rho\|_{L^1_tB_{2,1}^{\frac12} } } \big).
\end{equation}
The term   $ \|\rho\|_{L^1_tB_{2,1}^{\frac12}}$ will be controlled only by energy estimates. Indeed, the interpolation estimate 
$$
 \|\rho\|_{B_{2,1}^{\frac12}}\lesssim \|\rho\|_{L^2}^{\frac12}\|\nabla\rho\|_{L^2}^{\frac12}
$$
combined with  Proposition \ref{Energy}  and the   H\"{o}lder inequality give
\begin{eqnarray*}
 \|\rho\|_{L^1_tB_{2,1}^{\frac12}}&\lesssim& \|\rho\|_{L^\infty_t L^2}^{\frac12}t^{\frac34}\|\nabla\rho\|_{L^2_t L^2}^{\frac12}
 \lesssim t^{\frac34}\|\rho_0\|_{L^2}.
\end{eqnarray*}
 To control the term  $ \|x_h\rho\|_{L^1_tL^2}$  in the right hand side of \eqref{exp1},  we can  use Proposition \ref{prop07}: 
\begin{eqnarray*}
 \|x_h\rho\|_{L^1_tL^2}
 &\le&t\|x_h\rho\|_{L^\infty_t L^2}
\le C_0(1+t^{\frac94}).
\end{eqnarray*}
Consequently we obtain in view of \eqref{exp1}
\begin{equation}\label{eqs443}
\big\|\frac{\omega}{r}(t)\big\|_{L^{3,1}}\le C_0e^{C_0t^{\frac94}} e^{C\|\rho x_h\|_{L^1_tB_{\infty,1}^0}}.
\end{equation}
Now it remains to estimate the right term of \eqref{eqs443}  inside the exponential.  Let us first sketch the strategy of our approach. We will introduce an integer  $N(t)\in \NN$ that will be chosen in an optimal way in 
 the end and we will split  in frequency  the involved quantity into two parts: low frequencies corresponding to $q\le N(t)$ and high frequencies associated to $q>N(t).$ To estimate the low fequencies we use the dispersive result   of Proposition \ref{prop07}-(2). The estimate of high frequencies is based on a smoothing effect.

 By using Proposition \ref{prop07}-(2) and the  Bernstein inequality, we find that 
\begin{eqnarray}\label{log1}
 \nonumber\|x_h\rho\|_{L^1_tB_{\infty,1}^0}&=&\int_0^t\sum_{q\le N(\tau)}\|\Delta_q(x_h\rho)(\tau)\|_{L^\infty}d\tau+\int_0^t\sum_{q> N(\tau)}\|\Delta_q(x_h\rho)(\tau)\|_{L^\infty}d\tau\\
 &\le& C_0\int_0^t\big(\tau^{\frac14}+\tau^{-\frac34}\big)N(\tau)d\tau+C \int_0^t\sum_{q> N(\tau)}2^{q\frac32}\|\Delta_q(x_h\rho)(\tau)\|_{L^2}d\tau.
\end{eqnarray}
Now we intend to estimate the last sum in the above inequality.  For this purpose we  localize in frequency the  equation  for 
 $f=x_h\rho$  which is 
 $$
\partial_t f+v\cdot\nabla f-\Delta f=v^h\rho-2\nabla_h\rho:=F
$$
 By Setting $f_q:=\Delta_q f,$   we infer 
$$
\partial_t f_q+v\cdot\nabla f_q-\Delta f_q=-[\Delta_q,v\cdot\nabla] f+F_q
$$
From an  $L^2$ energy estimate, we obtain that 
$$
\frac{1}{2}\frac{d}{dt}\|f_q(t)\|_{L^2}^2-\int_{\RR^3}(\Delta f_q)f_qdx\le \|f_q\|_{L^2}\big(\|[\Delta_q,v\cdot\nabla] f\|_{L^2}+\|F_q\|_{L^2}\big).
$$
 Since the Bessel identity yields
$$
c2^{2q}\|f_q\|_{L^2}^2\le-\int_{\RR^3}(\Delta f_q)f_qdx, 
$$
 it  follows that
$$
\frac{d}{dt}\|f_q(t)\|_{L^2}+c2^{2q}\|f_q(t)\|_{L^2}\le C \big(\|[\Delta_q,v\cdot\nabla] f\|_{L^2}+\|F_q\|_{L^2}\big).
$$
Therefore we obtain   by integration in time that
$$
\|f_q(t)\|_{L^2}\lesssim e^{-ct2^{2q}}\|f_q(0)\|_{L^2}+\int_0^te^{-c(t-\tau)2^{2q}}\big(\|[\Delta_q,v\cdot\nabla] f\|_{L^2}+\|F_q\|_{L^2}\big)d\tau.
$$
To estimate the commutator in the right hand side,  we can use   Proposition \ref{propcomm} and \mbox{Proposition \ref{prop07}},  
\begin{eqnarray*}
\|[\Delta_q,v\cdot\nabla] f(\tau)\|_{L^2}&\le& C\|(\omega_{\theta}/r)(\tau)\|_{L^{3,1}}\big(\||x_h|^2\rho(\tau)\|_{L^6}+\|x_h\rho(\tau)\|_{L^2}\big)\\
&\leq& C_0 \|(\omega_{\theta}/r)(\tau)\|_{L^{3,1}}\big(\tau^{\frac{13}{6}}+\tau^{-\frac12}\big).
\end{eqnarray*}
Hence we get
\begin{eqnarray}\label{ha1}
\nonumber\|f_q(t)\|_{L^2}&\lesssim& e^{-ct2^{2q}}\|f_q(0)\|_{L^2}+\int_0^te^{-c(t-\tau)2^{2q}}\|F_q(\tau)\|_{L^2}d\tau\\&+&C_0\int_0^te^{-c(t-\tau)2^{2q}}
\|(\omega_{\theta}/r)(\tau)\|_{L^{3,1}}\big(\tau^{\frac{13}{6}}+\tau^{-\frac12}\big)d\tau.
\end{eqnarray}
Let us set  $\mathcal{K}(\tau)=\tau^{\frac{13}{6}}+\tau^{-\frac12},$ then   \eqref{ha1}  and convolution 
inequalities yield
\begin{eqnarray*}
&& \int_0^t\sum_{q>N(\tau)}2^{q\frac32}\|\Delta_q(x_h\rho)(\tau)\|_{L^2}d\tau\lesssim\sum_{q\geq-1}
2^{-\frac12 q}\big(\|f_q(0)\|_{L^2}+\|F_q\|_{L^1_tL^2}\big)\\
\quad&&\quad+C_0 \int_0^t\sum_{q> N(\tau)}2^{q\frac32}\int_0^\tau e^{-c(\tau-\tau')2^{2q}}\mathcal{K}(\tau')
\|(\omega_{\theta}/r)(\tau')\|_{L^{3,1}}d\tau'\\
\quad&&\quad\lesssim\|f_0\|_{L^2}+\|F\|_{L^1_tL^2}+ C_0\int_0^t\|\omega_{\theta}/r\|_{L^\infty_\tau L^{3,1}}\Big(\sum_{q> N(\tau)}2^{q\frac32}\int_0^\tau e^{-c(\tau-\tau')2^{2q}}\mathcal{K}(\tau')
d\tau'\Big)d\tau.
\end{eqnarray*}
Moreover, from  Proposition \ref{Energy} and Lemma \ref{Nashlem}, we also have 
\begin{eqnarray*}
\|F\|_{L^1_tL^2}&\le& t^{\frac12}\|\nabla\rho\|_{L^2_tL^2}+\|v\|_{L^\infty_tL^2}\|\rho\|_{L^1_tL^\infty}\\
&\le&C\|\rho_0\|_{L^2} t^{\frac12}+C_0(1+t)\|\rho_0\|_{L^2}\Big(\int_0^t\tau^{-\frac{3}{4}}d\tau+t\Big)\\
&\le&C_0(1+t^{2}).
\end{eqnarray*}
Inserting these estimates into \eqref{log1} yields
\begin{eqnarray}\label{log200}
 \nonumber\|x_h\rho\|_{L^1_tB_{\infty,1}^0}
&\le& C_0(1+t^{2})+C_0\int_0^t\big(\tau^{\frac14}+\tau^{-\frac34}\big)N(\tau)d\tau\\
\nonumber&+&C_0
   \int_0^t\|\omega_{\theta}/r\|_{L^\infty_\tau L^{3,1}}\Big(\sum_{q>N(\tau)}2^{q\frac32}\int_0^\tau e^{-c(\tau-\tau')2^{2q}}\big(\{\tau'\}^{\frac{13}{6}}+\{\tau'\}^{-\frac12}\big)d\tau'\Big)d\tau\\
   \nonumber&\le&
   C_0(1+t^{2})+C_0\int_0^t\big(\tau^{\frac14}+\tau^{-\frac34}\big)N(\tau)d\tau\\
   &+&C_0\int_0^t\|\omega_{\theta}/r\|_{L^\infty_\tau L^{3,1}}\Big(\tau^{\frac{13}{6}}2^{-\frac12 N(\tau)}+\sum_{q>N(\tau)}2^{q\frac32}\int_0^\tau e^{-c(\tau-\tau')2^{2q}}\{\tau'\}^{-\frac12}d\tau'\Big)d\tau.
\end{eqnarray}
By a change of variables we get
\begin{eqnarray*}
\\ & &  \sum_{q> N(\tau)}2^{q\frac32}\int_0^\tau e^{-c(\tau-\tau')2^{2q}}\{\tau'\}^{-\frac12}d\tau' \\ 
&=&\sum_{q> N(\tau)}2^{q\frac12}e^{-c\tau 2^{2q}}\int_0^{2^{2q}\tau} e^{c\tau'}\{\tau'\}^{-\frac12}d\tau'\\
&=&
\sum_{q\in \Lambda_1(\tau)}2^{q\frac12}e^{-c\tau 2^{2q}}\int_0^{2^{2q}\tau} e^{c\tau'}\{\tau'\}^{-\frac12}d\tau'
+\sum_{q\in \Lambda_2(\tau)}2^{q\frac12}e^{-c\tau 2^{2q}}\int_0^{2^{2q}\tau} e^{c\tau'}\{\tau'\}^{-\frac12}d\tau'\\
&:=&\hbox{I}(\tau)+\hbox{II}(\tau).
\end{eqnarray*}
with 
$$\Lambda_1(\tau)=\Big\{ q> N(\tau)\hbox{ and } \tau 2^{2q}\geq1\Big\}\quad\hbox{and}\quad \Lambda_2(\tau)= \Big\{ q> N(\tau)\hbox{ and } \tau 2^{2q}\leq1\Big\}.
$$ To estimate the first term we use the following inequality which can, be proven by integration by parts: there exists $C>0$ such that for every $x\geq 1$ 
$$
\int_0^xy^{-\frac12}e^{cy}dy\le Cx^{-\frac12}e^{cx}.
$$
It follows that
\begin{eqnarray*}
\hbox{I}(\tau)&\lesssim&\tau^{-\frac12}\sum_{q> N(\tau)}2^{-\frac12 q}
\lesssim 2^{-\frac12N(\tau)}\tau^{-\frac12}.
\end{eqnarray*}
To estimate the second term, we observe that the integral is bounded by  a fixed number and
 hence, we find that 
\begin{eqnarray*}
\hbox{II}(\tau)&\lesssim&\sum_{q\in \Lambda_2(\tau)}2^{\frac12 q}
\lesssim 2^{-\frac12N(\tau)}\sum_{2^{2q}\leq\tau^{-1}}2^{ q}
\lesssim 2^{-\frac12N(\tau)}(1+\tau^{-\frac12}).
\end{eqnarray*}
Gathering  these estimates, we obtain
$$
\sum_{q>N(\tau)}2^{q\frac32}\int_0^\tau e^{-c(\tau-\tau')2^{2q}}\{\tau'\}^{-\frac12}d\tau'\le 2^{-\frac12N(\tau)}(1+\tau^{-\frac12}).
$$
 By plugging  this estimate into \eqref{log200}, we get 
\begin{eqnarray}
 \nonumber\|x_h\rho\|_{L^1_tB_{\infty,1}^0}
&\le& C_0(1+t^{2})+C_0\int_0^t\big(\tau^{\frac{13}{6}}+\tau^{-\frac34}\big)\Big(N(\tau)+\|\omega_{\theta}/r\|_{L^\infty_\tau L^{3,1}}2^{-\frac12 N(\tau)}\Big)d\tau.
\end{eqnarray}
We choose $N$ such that
$$
N(\tau)=2\Big[\log_2\big(2+\|\omega_{\theta}/r\|_{L^\infty_\tau L^{3,1}}\big)\Big]
$$
and then we find
\begin{equation}\label{eqs444}
\|x_h\rho\|_{L^1_tB_{\infty,1}^0}
\le C_0(1+t^{\frac{19}{6}})+C_0\int_0^t\big(\tau^{\frac{13}{6}}+\tau^{-\frac34}\big)\log\big(2+\|\omega_{\theta}/r\|_{L^\infty_\tau L^{3,1}}\big)d\tau.
\end{equation} \
Putting together \eqref{eqs443}, \eqref{eqs444} and  Proposition \ref{prop07}-(2), we find that 
\begin{eqnarray*}
\log\big(2+\|\omega_{\theta}/r\|_{L^\infty_t L^{3,1}}\big)&\le& C_0(1+t^{\frac{19}{6}})+C_0\int_0^t\big(\tau^{\frac{13}{6}}+\tau^{-\frac34}\big)\log\big(2+\|\omega_{\theta}/r\|_{L^\infty_\tau L^{3,1}}\big)d\tau.
. \end{eqnarray*}
From the  Gronwall inequality, we infer 
\begin{eqnarray*}
\log\big(2+\|\omega_{\theta}/r\|_{L^\infty_t L^{3,1}}\big)\le C_0(1+t^{\frac{19}{6}})e^{C_0(t^{\frac{19}{6}}+t^{\frac14})}
 \le \Phi_2(t).
\end{eqnarray*}
Therefore we get by using again \eqref{eqs444} that 
$$
\big\|{\omega_\theta \over r}(t)\big\|_{L^{3,1}}\le \Phi_2(t).
$$
Since $\frac{\omega}{r}=\frac{\omega_\theta}{r} e_\theta$,  \eqref{prop1-2} implies that
\begin{equation}\label{vort54}
\big\|{\omega \over r}(t)\big\|_{L^{3,1}}\le\Phi_2(t).
\end{equation}
Finally, thanks to \eqref{convol}, we obtain
\begin{eqnarray*}
\|\frac{v^r}{r}(t)\|_{L^\infty}\le C\|\frac{\omega}{r}(t)\|_{L^{3,1}}
\le \Phi_2(t).
\end{eqnarray*}
This ends  the proof of Proposition \ref{Strong}.
\end{proof}
\subsubsection{Estimate of $\|\omega(t)\|_{L^\infty}$} Our purpose now is to bound the vorticity.
\begin{prop}\label{vorticity}
Let $v_0\in L^2$ be an axisymmetric divergence free vector field without swirl  such that $\omega_0\in L^\infty,\,\frac{\omega_0}{r}\in L^{3,1}$. Let  $\rho_0$ be axisymmetric scalar function, belonging to $ L^2\cap L^{m}, m>6$ and such that $|x_h|^2\rho_0\in L^2$. Then we have for every $t\in\RR_+$
$$
\|\omega(t)\|_{L^\infty}+\|\nabla\rho\|_{L^1_tL^\infty}\le \Phi_4(t).
$$
\end{prop}
\begin{proof}
 From   the maximum principle  for   the equation \eqref{tourbillon}, we obtain that 
$$
\|\omega(t)\|_{L^\infty}\le\|\omega_0\|_{L^\infty}+\int_0^t\|v^r/r(\tau)\|_{L^\infty}\|\omega(\tau)\|_{L^\infty}d\tau+\int_0^t\|\nabla\rho(\tau)\|_{L^\infty}d\tau.
$$
 By combining Proposition \ref{Strong}  and the  Gronwall inequality, this yields 
$$
\|\omega(t)\|_{L^\infty}\le\Phi_3(t)\Big( 1+\int_0^t\|\nabla\rho\|_{L^\infty}d\tau\Big).
$$
Now we claim that, 
\begin{equation}\label{smooth7}
\|\nabla\rho\|_{L^1_tL^\infty}\le C_0\Big(1+t^{2}+\int_0^t\|\omega(\tau)\|_{L^\infty}d\tau\Big).
\end{equation}
Let us first finish the proof by using this  estimate.
 We deduce that 
$$
\|\omega(t)\|_{L^\infty}\le \Phi_3(t)\Big( 1+\int_0^t\|\omega(\tau)\|_{L^\infty}d\tau\Big).
$$
and thanks to  the Gronwall inequality that 
$$
\|\omega(t)\|_{L^\infty}\le \Phi_4(t).
$$
This gives in turn
$$
\|\nabla\rho\|_{L^1_tL^\infty}\le \Phi_4(t).
$$
Let us now come back to the proof of \eqref{smooth7}.  For $q\in \NN$ we set $\rho_q:=\Delta_q\rho,$ then
$$
\partial_t\rho_q+v\cdot\nabla\rho_q-\Delta\rho_q=-[\Delta_q,v\cdot\nabla]\rho
$$
Let $p\geq2$ then multiplying this equation by $|\rho_q|^{p-2}\rho_q$ and using H\"{o}lder inequality
$$
\frac{1}{p}\frac{d}{dt}\|\rho(t)\|_{L^p}^p-\int_{\RR^3}(\Delta\rho_q)|\rho_q|^{p-2}\rho_qdx\le\|\rho_q\|_{L^p}^{p-1}\big\|[\Delta_q,v\cdot\nabla]\rho\big\|_{L^p}.
$$
Now we use the generalized Bernstein inequality, see \cite{Lemar},
$$
\frac1p\, 2^{2q}\|\rho_q\|_{L^p}^p\le-\int_{\RR^3}(\Delta\rho_q)|\rho_q|^{p-2}\rho_qdx.
$$
Hence we get
$$
\frac{d}{dt}\|\rho(t)\|_{L^p}+c_p2^{2q}\|\rho_q\|_{L^p}\lesssim \big\|[\Delta_q,v\cdot\nabla]\rho\big\|_{L^p}.
$$
This gives 
\begin{equation}\label{diff1}
\|\rho_q(t)\|_{L^p}\le e^{-c_pt2^{2q}}\|\Delta_q\rho_0\|_{L^p}+\int_0^te^{-c_p2^{2q}(t-\tau)}\big\|[\Delta_q,v\cdot\nabla]\rho\big\|_{L^p}d\tau.
\end{equation}
Integrating in time implies that
$$
\|\rho_q\|_{L^1_tL^p}\lesssim 2^{-2q}\|\rho_0\|_{L^p}+2^{-2q}\big\|[\Delta_q,v\cdot\nabla]\rho\big\|_{L^1_tL^p}
$$
According to Proposition 2.3 \cite{hk} and Proposition \ref{Energy} we have
\begin{eqnarray*}
\big\|[\Delta_q,v\cdot\nabla]\rho\big\|_{L^p}&\le& C\|\rho\|_{L^p}\big((q+1)\|\omega\|_{L^\infty}+\|\nabla\Delta_{-1}v\|_{L^\infty}\big)\\
&\le&C\|\rho_0\|_{L^p}\big((q+1)\|\omega\|_{L^\infty}+\|v\|_{L^2}\big)\\
&\le&C\|\rho_0\|_{L^p}\big((q+1)\|\omega\|_{L^\infty}+C_0(1+t)\big).
\end{eqnarray*}
It follows that
\begin{eqnarray*}
\|\rho_q\|_{L^1_tL^p}\le C_0(1+t^2)2^{-2q}+C\|\rho_0\|_{L^p}(q+1)2^{-2q}\|\omega\|_{L^1_tL^\infty}.
\end{eqnarray*}
By using the  Bernstein inequality, we find  for $p>3$ that 
\begin{eqnarray*}
\|\nabla\rho\|_{L^1_tL^\infty}&\le&Ct\|\rho_0\|_{L^2}+C\sum_{q\in\NN}2^{q(1+\frac3p)}\|\rho_q\|_{L^1_tL^p}\\
&\le& C_0(1+t^2)\sum_{q\in\NN}2^{q(-1+\frac3p)}+C_0\|\omega\|_{L^1_tL^\infty}
\sum_{q\in\NN}2^{q(-1+\frac3p)}(q+1)\\
&\le&C_0(1+t^2)+C_0\|\omega\|_{L^1_tL^\infty}.
\end{eqnarray*}
This ends  the proof of the desired inequality.
\end{proof}
\subsubsection{Lipschitz bound of the velocity} We shall  now deal with the global propagation  of the sub-critical Sobolev regularities. This is basically related to the control of the Lipschitz norm of the velocity. 
\begin{prop}\label{Soblev-est}
Let $\frac52<s<3$ and $(v_0,\rho_0)\in H^s\times H^{s-2}$ and $(v,\rho)$ be a solution of the Boussinesq system \eqref{bsintro}. Then we have for every $t\geq0$
$$
\|v\|_{\widetilde L^\infty_tH^s}+\|\rho\|_{\widetilde L^\infty_tH^{s-2}}+\|\rho\|_{\widetilde L^1_tH^{s}}\le C_0(1+t)e^{C\|\nabla v\|_{L^1_tL^
\infty}}.$$
If in addition $\rho_0\in L^m$ with $ m>6$ and $|x_h|^2\rho_0\in L^2,$ then we get for every $t\geq0$
$$
\|\nabla v(t)\|_{L^\infty}\le \Phi_5(t);\quad\quad \|v\|_{\widetilde L^\infty_tH^s}+\|\rho\|_{\widetilde L^\infty_tH^{s-2}}+\|\rho\|_{\widetilde L^1_tH^{s}}\le \Phi_6(t).
$$
\end{prop}
\begin{Rema}
We point out that we can extend the results of  Proposition \ref{Soblev-est} to higher regularities $s\geq3$ but for the sake of  simplicity  we restrict ourselves here to the case of $s<3$. 
\end{Rema}
\begin{proof}
We localize in frequency the equation of the velocity. For $q\in \NN\cup\{-1\}$ we set $v_q:=\Delta_q v$ and $\rho_q:=\Delta_q\rho.$
$$
\partial_t v_q+v\cdot\nabla v_q+\nabla\pi_q=\rho_q e_z-[\Delta_q,v\cdot\nabla]v.
$$
Thus taking the $L^2$-scalar product with $v_q$ and using the incompressibility of $v$ and $v_q$  we get
$$
\frac{d}{dt}\|v_q(t)\|_{L^2}\le\|\rho_q\|_{L^2}+\|[\Delta_q,v\cdot\nabla]v\|_{L^2}.
$$
Integrating in time we obtain
$$
\|v_q(t)\|_{L^2}\le\|v_q(0)\|_{L^2}+\|\rho_q\|_{L^1_tL^2}+\|[\Delta_q,v\cdot\nabla]v\|_{L^1_tL^2}.
$$
Thus we get
$$
\|v\|_{\widetilde L^\infty_t H^s}\le \|v_0\|_{H^s}+\|\rho\|_{\widetilde L^1_t H^s}+\|(2^{qs}\|[\Delta_q,v\cdot\nabla]v\|_{L^1_tL^2})_{q}\|_{\ell^2}.
$$
We will use the commutator estimate, see for instance Lemma B.5 of  \cite{rdd},
$$
\|(2^{qs}\|[\Delta_q,v\cdot\nabla]v\|_{L^1_tL^2})_{q}\|_{\ell^2}\le C\int_0^t\|\nabla v(\tau)\|_{L^\infty}\|v(\tau)\|_{H^s}d\tau.
$$
Putting together these estimates and using Gronwall inequality yield
\begin{equation}\label{dif9}
\|v\|_{\widetilde L^\infty_t H^s}\le \big(\|v_0\|_{H^s}+\|\rho\|_{\widetilde L^1_t H^s}\big)e^{C\|\nabla v\|_{L^1_tL^\infty}}.
\end{equation}
Using the estimate \eqref{diff1} we get for $q\in\NN$
$$
\|\rho_q\|_{L^\infty_tL^2}+2^{2q}\|\rho_q\|_{L^1_tL^2}\le C\|\rho_q(0)\|_{L^2}+\|[\Delta_q,v\cdot\nabla]\rho\|_{L^1_tL^2}.
$$
Therefore we find
\begin{eqnarray*}
\|\rho\|_{\widetilde L^\infty_t H^{s-2}}+\|\rho\|_{\widetilde L^1_t H^s}&\le&\|\Delta_{-1}\rho\|_{L^1_tL^2}+ \|\rho_0\|_{H^{s-2}}+\|(2^{q(s-2)}\|[\Delta_q,v\cdot\nabla]\rho\|_{L^1_tL^2})_{q}\|_{\ell^2}\\
&\le& Ct\|\rho_0\|_{L^2}+ \|\rho_0\|_{H^{s-2}}+\|(2^{q(s-2)}\|[\Delta_q,v\cdot\nabla]\rho\|_{L^1_tL^2})_{q}\|_{\ell^2}.
\end{eqnarray*}
Since $-1<s-2<1$ then we have the estimate, see \cite{rdd}, 
$$
\|(2^{q(s-2)}\|[\Delta_q,v\cdot\nabla]\rho\|_{L^1_tL^2})_{q}\|_{\ell^2}\le C\int_0^t\|\nabla v(\tau)\|_{L^\infty}\|\rho(\tau)\|_{H^{s-2}}d\tau.
$$
Consequently,
$$
\|\rho\|_{\widetilde L^\infty_t H^{s-2}}+\|\rho\|_{\widetilde L^1_t H^s}\le C_0(1+t)+C\int_0^t\|\nabla v(\tau)\|_{L^\infty}\|\rho(\tau)\|_{H^{s-2}}d\tau.
$$
By Gronwall inequality
\begin{equation}\label{dif88}
\|\rho\|_{\widetilde L^\infty_t H^{s-2}}+\|\rho\|_{\widetilde L^1_t H^s}\le C_0(1+t)e^{C\|\nabla v\|_{L^1_tL^\infty}}.
\end{equation}
Combining this estimate with \eqref{dif9} gives the desired estimates.

Now to get a global bound for Lipschitz norm of the velocity we use the classical logarithmic estimate: for $s>\frac32$
$$
\|\nabla v\|_{L^\infty}\lesssim \|v\|_{L^2}+ \|\omega\|_{L^\infty}\log(e+\|v\|_{H^s}).
$$
Combining this estimate with the first result of Proposition \ref{Soblev-est} and Proposition \ref{vorticity}
$$
\|\nabla v\|_{L^\infty}\le \Phi_4(t)\Big(1+\int_0^t\|\nabla v(\tau)\|_{L^\infty}d\tau\Big).
$$
It follows from Gronwall inequality that
$$
\|\nabla v(t)\|_{L^\infty}\le \Phi_5(t).
$$
Plugging this estimate into \eqref{dif9} and  \eqref{dif88} gives
$$
\|v\|_{\widetilde L^\infty_t H^s}+\|\rho\|_{\widetilde L^\infty_t H^{s-2}}+\|\rho\|_{\widetilde L^1_t H^s}\le \Phi_6(t).
$$ 
This ends the proof of the proposition. 
\end{proof}
\section{Proof of the main result}
\label{sectionproof}
The proof of the existence part of Theorem \ref{thm1} can be done in a classical way by smoothing out the initial data as follows
$$
v_{0,n}=S_nv_0=2^{3n}\chi(2^n\cdot)\star v_0,\,\rho_{0,n}=S_n\rho_0=2^{3n}\chi(2^n\cdot)\star\rho_0
$$
where $S_n$ is the cut-off in frequency defined in the preliminaries. Since $\chi$ is radial then the functions $v_{0,n}$ and $\rho_{0,n}$ remain axisymmetric. Moreover this family is uniformly bounded in the space of initial data: this is obvious in  Sobolev and Lebesgue spaces but it remains to check the uniform boundedness of the horizontal moment of the density. We will show that 
\begin{equation}\label{eq321}
\sup_{n\in\NN}\||x_h|^2\rho_{n,0}\|_{L^2}\le C\big(\|\rho_0\|_{L^2}+\||x_h|^2\rho_{0}\|_{L^2}\big).
\end{equation}
For this purpose we write
\begin{eqnarray*}
|x_h|^2|\rho_{n,0}(x)|&=&|x_h|^2\Big|\int_{\RR^3}\chi(2^{n}(x-y))\rho(y)dy\Big|\\
&\le&2\int_{\RR^3}|x_h-y_h|^2|\chi|(2^{n}(x-y))|\rho|(y)dy+2\Big|\int_{\RR^3}|\chi|(2^{n}(x-y))|y_h|^2\rho(y)dy\Big|\\
&\le&2\,2^{-2n}(2^{3n}\chi_1(2^n\cdot)\star\rho)(x)+2(|\chi|(2^n\cdot)\star(|y_h|^2\rho))(x)
\end{eqnarray*}
with $\chi_1(x)=|x_h|^2|\chi(x).$ From convolution laws we get
$$
\||x_h|^2|\rho_{n,0}\|_{L^2}\le C2^{-2n}\|\rho\|_{L^2}+C\||x_h|^2\rho_0\|_{L^2}.
$$
This achieves the proof of \eqref{eq321}. Now, by using standard arguments based on   the a priori estimates described in  Proposition \ref{Soblev-est}, Proposition \ref{vorticity}  and Proposition \ref{prop07} we can construct  a unique global solution $(v_n,\rho_n)$  in the following space
$$
v_n\in\mathcal{C}(\RR_+; H^s)\cap L^1(\RR_+;W^{1,\infty})\quad \hbox{and}\quad\rho_n\in \mathcal{C}(\RR_+; H^{s-2})\cap L^1(\RR_+;W^{1,\infty}).
$$
The control is uniform with respect to the parameter $n$. Therefore we can prove the  strong convergence of 
 a subsequence of  $(v_n,\rho_n)_{n\in\NN}$  to  some  $(v,\rho)$  belonging to the same space and satisfying the initial value problem. It remains to prove the uniqueness problem. This gives the existence of a  solution.

 The uniqueness will be proven in the following space
$$
(v,\rho)\in\mathcal{X}:= \Big(\mathcal{C}(\RR_+;L^2)\cap L^{1}(\RR_+; W^{1,\infty})\Big)^2.
$$

Let $(v^i,\rho^i)\in \mathcal{X}, 1\leq i\leq 2$ be  two solutions of the
 system \eqref{bsintro} with the same initial \mbox{data $(v_0,\theta_0)$} and denote
 $\delta v=v^2-v^1,\delta\rho=\rho^2-\rho^1$. Then
\begin{equation}\label{diff}
\left\{ \begin{array}{ll}
\partial_t\delta v+v^2\cdot\nabla\delta v+\nabla\Pi=-\delta v\cdot\nabla v^1+\delta\rho\, e_z
\\
\partial_t \delta\rho +v^2\cdot\nabla\delta\rho -  \Delta\,  \delta \rho=
 -\delta v\cdot\nabla \rho^1
\\
{\mathop{\rm div}}\,  v^i=0.
\end{array}
\right.
\end{equation}
Taking the $L^2$-scalar product of the first equation with $\delta v$ and integrating by parts gives
$$
\frac12\frac{d}{dt}\|\delta v(t)\|_{L^2}^2\le\|\nabla v^1\|_{L^\infty}\|\delta v\|_{L^2}^2+\|\delta\rho\|_{L^2}\|\delta v\|_{L^2}.
$$
Consequently,
$$
\frac{d}{dt}\|\delta v(t)\|_{L^2}\le\|\nabla v^1\|_{L^\infty}\|\delta v\|_{L^2}+\|\delta\rho\|_{L^2}.
$$
Using  the Gronwall inequality yields
$$
e^{-\|\nabla v^1\|_{L^1_tL^\infty}}\|\delta v(t)\|_{L^2}\le \Big(\|\delta v_0\|_{L^2}+\int_0^te^{-\|\nabla v^1\|_{L^\tau_tL^\infty}}\|\delta\rho(\tau)\|_{L^2}d\tau\Big).
$$
By the same computations we get
$$
\|\delta \rho(t)\|_{L^2}\le\|\delta \rho_0\|_{L^2}+\int_0^t\|\nabla\rho^1(\tau)\|_{L^\infty}\|\delta v(\tau)\|_{L^2}d\tau.
$$
It suffices now to put together these estimates and to use  the  Gronwall inequality.

\newpage

\appendix

\section{Nash-De Giorgi estimates for convection-diffusion equations}
\label{appendix}
\begin{lem}
\label{Nashlem}
 Consider the equation 
 \begin{equation}
 \label{td}
 \partial_{t}f + u  \cdot \nabla  f - \Delta f = \nabla \cdot F + G , \quad  t>0, \quad x \in \mathbb{R}^3, \quad f(0, x)= f_{0}(x).
 \end{equation}
   Consider   $p,\,  q, \, p_{1}, \, q_{1},  \in [1, +\infty], r\in[2,+\infty]$  with 
   $${2 \over p}+ {3 \over q}<1,\quad{ 2 \over p_{1} }+ {3 \over q _{1}} <2.
  $$
   There exists $C>0$ such that  for  every smooth divergence free vector field $u$,
    for every $F\in L^p_{T} L^q$ and for every $f_{0} \in L^r$,  the solution of \eqref{td} satisfies
     the estimate: for every $t\in]0,T],$
   \begin{eqnarray}
   \label{Nash}
  \nonumber  \|f(t)  \|_{L^\infty} &\leq& C \Big( 1 + {1 \over t^{ 3  \over 2 r } }\Big)  { \|f_{0} \|_{L^r}  } +C \Big( 1 +  \sqrt{T}^{ 1- ( {2 \over p } + {3 \over q} ) } \Big)  { \|F \|_{L^p_{T} L^q}  } \\&+& C \Big( 1 +  \sqrt{T}^{ 2- ( {2 \over p_{1} } + {3 \over q_{1}} ) } \Big) 
      { \|G \|_{L^{p_{1}}_{T} L^{q_{1} } }  } .
   \end{eqnarray}
   \end{lem}
  \begin{proof}
  Since the equation is linear, we can study separately the three problems
   \begin{equation}
   \label{td*} 
 \left\{\begin{array}{ll} \mathcal{P}f= \nabla \cdot F, \\
  f(0, x)=0 \end{array}\right., \quad
 \left\{\begin{array}{ll} \mathcal{P}f=  G, \\
  f(0, x)=0 \end{array}\right., \quad
  \left\{\begin{array}{ll} \mathcal{P}f=0, \\
  f(0, x)=f_{0}(x) \end{array}\right.,
   \end{equation}
   where we have set $\mathcal{P}f=   \partial_{t}f + u  \cdot \nabla  f - \Delta f $.
   
   Let us start with the first problem in \eqref{td*}. We shall prove that
    there exists $C>0$ such that 
    for every $F$  with  $ \|F \|_{L^p_{1} L^q} \leq 1$, we have the estimate
    \begin{equation}
    \label{T=1}
     \|f \|_{L^\infty_{1}L^\infty} \leq C.
    \end{equation}
  Once this estimate, is proven, the estimate involving $F$ in \eqref{Nash} will just follow by
   a scaling argument.
   
    The first step is to   use  the  standard $L^q$ a priori estimate (obtained by multiplying
    by  
    $\mathcal{P} f $  by $ |f|^{q-1} \mbox{sign }f$). Since $u$
     is divergence free, we have that 
 $$    {d \over dt } \Big(  {1 \over q} \|  f(t) \|_{L^q}^q  \Big)+  (q- 1)  \int_{\mathbb{R}^2} |\nabla f |^2 |f|^{q-2}\, dx
     \leq   (q-1)   \int_{\mathbb{R}^3 } | F |   \,  |\nabla f |\, | f |^{q-2} \, dx.$$
       From the Young inequality and the Holder inequality (note that since $2/p + 3/q <1$, 
        we necessarily  have that $q>3$ and $p>2$) this yields
  $$  {d \over dt } \Big(  {1 \over q} \|  f(t) \|_{L^q}^q  \Big) \leq  (q-1) \|F \|_{L^q}^2 \, \|f \|_{L^q}^{q-2}$$
   and hence by integration in time, we obtain since $q>3$ that 
      \begin{equation}
      \label{L2f}  \|f\|_{L^{\infty}_{1}L^q} \leq C_{q} \Big( \int_{0}^1   \| F(t) \|_{L^q}^2 dt \Big)^{1 \over 2} \leq
      C_{q}   \| F\|_{L^p_{1}L^q} \leq C_{q}
        \end{equation}
        where  $C_{q}$ depends only on $q$.
        To improve this estimate that is to go from the above $L^q$ estimate to an $L^\infty$
         estimate, we shall  follow the  De-Giorgi, Nash iteration argument.
    For $M>0$ to be chosen, let us take  a positive  increasing sequence $(M_{k})_{k \geq 0}$ 
     such that  $M_{k} \leq M$
      and  $M_{k}$ converges towards $M$.   A good choice is  for example
     \begin{equation}
     \label{Mk}
      M_{k}= M( 1 -{ 1 \over k+ 1  } ).
     \end{equation}
       We shall use the standard notation $x_{+} = \mbox{max}(x, 0)$.
     Since $u$ is divergence free,  we obtain  the   level set
      energy estimate
   \begin{eqnarray*} {d \over dt } \Big(  {1 \over 2} \| (f -M_{k})_{+}(t) \|_{L^2}^2  \Big)+ \| \nabla( f - M_{k})_{+} \|_{L^2}^2
   &  \leq  &   \int_{ f \geq M_{k}} | F |  \,  |\nabla f | \, dx \\
    &  \leq  &  \Big(  \int_{ f \geq M_{k} } |F|^2 \Big)^{1 \over 2} \| \nabla ( f- M_{k})_{+}\|_{L^2} 
    \end{eqnarray*}
    where the last inequality comes from Cauchy-Schwarz. By using the Young inequality, 
     we thus obtain
     \begin{equation}
     \label{Uk}
       U_{k} \leq    \int_{0}^1 \int_{ f \geq M_{k} } |F|^2  \, dx dt
      \end{equation}
    where 
    $$ U_{k}=  \| (f -M_{k})_{+} \|_{L^\infty_{1} L^2}^2 +  \| \nabla( f -M_{k})_{+} \|_{L^2_{1} L^2}^2.$$
    The main idea is to prove that the right-hand side of \eqref{Uk} can estimated
     by  a power of  $ U_{k-1}$ strictly larger than $1$. 
      By using the Holder inequality, we first get that
      \begin{equation}
      \label{Uk1}  U_{k} \leq  \int_{0}^1  \,  \| F (t)  \|_{L^q}^2   m_{k}(t)^{  1 - {2 \over q } }\, dt
       \leq  \| F \|_{L^p_{1}L^q}^2  \Big( \int_{0}^1 m_{k}(t)^{  (1 - {2 \over q} ) ({p \over p-2})}  \, dt \Big)
       ^{ 1 - {2 \over p}  }.\end{equation}
       where $m_{k}(t) = | \{x, \,  f(t) \geq M_{k} \}  |.$
     To estimate $m_{k}(t)$, we note that   if $f(t, x) \geq M_{k}$  then
     $$ f(t, x)  -M_{k - 1} \geq M_{k} - M_{k-1} \geq 0$$ and thus we have
      \begin{equation}
      \label{mk0} {\bf 1}_{f(t,x) \geq M_{k} } \leq  {   (k+ 1) ^2 \over M  } (f(t,x) -M_{k-1})_{+}.
      \end{equation}  
  This yields
  \begin{equation}
  \label{mk}
   m_{k}(t) \leq   { (k+ 1) ^{2 m } \over M^m} \|(f(t) -M_{k-1})_{+} \|_{L^m}^m
  \end{equation} 
  for every $m \geq 1$. We shall choose $m$ carefully below.
  By plugging this last estimate in  \eqref{Uk1}, we get
 \begin{equation}
 \label{Uk2} U_{k} \leq  \| F \|_{L^p_{1}L^q}^2   \Big({(k+ 1)^2 \over M }\Big)^{m( 1 - {2 \over p } )}\Big(  \int_{0}^1    \|(f(t) -M_{k-1})_{+} \|_{L^m}^{m(  1 - { 2 \over q }) ( { p \over p-2} )  } \Big)^{ 1 - {2 \over p } }\end{equation}
  Now let us notice that   if $\alpha \geq 1 $ and $\beta \in [2, 6]$ are such that 
  $   {2 \over \alpha}+ { 3 \over \beta } \geq {3 \over 2 }$ then we have
  $$ \|(f(t) -M_{k-1})_{+} \|_{L^\alpha_{1}L^\beta}^2 \leq  U_{k-1}.$$
  Indeed  the control of $U_{k-1}$ gives a control of the $L^\infty_{1}L^1$ and
   the $L^2_{1}H^1$ norm. By Sobolev embedding this  gives a control
    of the $L^2_{1}L^6$ norm and then   
     the inequality  follows by standard interpolation in Lebesgue spaces. 
  Consequently, to achieve our program,  we need to choose $m \in [2, 6]$ such that
  \begin{equation}
  \label{contrainte} m ( 1 - { 2 \over q } )  >2, \quad  2  { { 1 - {2 \over p } }  \over m \big( 1 -  {2 \over q}  \big) }
   + {3 \over m } \geq {3 \over 2 }.
   \end{equation}
   The first constraint  can be satisfied as soon as
    $2 / (1 - 2/q)<6$ which is equivalent to $q>3$ while the second constraint can be satisfied as soon as
     $$  2 \big(  1 - {2 \over p }  \big) + 3( 1 - {2 \over q }) \geq m ( 1 - { 2 \over q } )  >  { 3\over 2 }\cdot  2= 3$$
 which is equivalent to 
 $$ { 2 \over p } + { 3 \over q } <1.$$
  Consequently,   since we have  $q>3$ and $ 2/p + 3/q<1$  by assumption
   we can  choose $m$ such that  the constraint  \eqref{contrainte} are matched.
    This yields that there exists $\gamma >1$ such that
    $$  U_{k} \leq  \| F \|_{L^p_{1}L^q}^2   \Big({(k+ 1)^2 \over M }\Big)^{m( 1 - {2 \over p } )} U_{k-1}^\gamma
     \leq   \Big({(k+ 1)^2 \over M }\Big)^{m( 1 - {2 \over p } )} U_{k-1}^\gamma, \quad \forall k \geq 1.$$
      If  $U_{0}$ is   sufficiently small, this yields that $ \lim_{k \rightarrow + \infty } U_{k}= 0$.
      Since, we have  from \eqref{Uk1},  the  Tchebychev inequality and the energy inequality  \eqref{L2f}
       that
     $$ U_{0} \leq   \Big( \int_{0}^1 m_{0}(t)^{  (1 - {2 \over q} ) ({p \over p-2})}  \, dt \Big)
       ^{ 1 - {2 \over p}  } \leq    \Big({ \| f \|_{L^\infty_{1} L^q} \over M-1} \Big)^
       {   q - {2  } } \leq  \Big({  C_{q}  \over M-1} \Big)^
       {  q - {2  } },  $$
        we can indeed  make $U_{0}$ arbitrarily small by taking 
         $M$ sufficiently large and thus $\lim_{k \rightarrow + \infty} U_{k}=0$. From Fatou's Lemma, we obtain that
          for every $t \in [0, 1]$
          $$ \int_{\mathbb{R}^3} (f(t,x) -M)_{+} \, dx \leq 0$$
           and therefore that almost everywhere 
         $$ f(t, x) \leq M.$$
      By changing $f$ into $-f$,we obtain in a similar way that $f \geq -  M$ almost everywhere
       and thus \eqref{T=1} is proven.
        To obtain  the part of estimate \eqref{Nash} involving $F$ if $T \geq 1$
         we can use a change of  scale  argument. Let us set
        $  K  \tilde{f}(\tau,X)= f(  T \,\tau ,\sqrt{T}\, X )$
         for $K>0$ to be chosen. Then  we have $ K\|\tilde{f}\|_{L^\infty_{1}L^\infty} = \| f \|_{L^\infty_{T} L^\infty}$ and $\tilde{f}(\tau, X)$ solves the equation
         $$ \partial_{\tau} \tilde{f} + \tilde{u} \cdot \nabla_{X} \tilde{f} - \Delta_{X} \tilde{f}= \nabla_{X} \cdot  \tilde{F}$$
          where $\tilde{u}$ is still divergence free and 
          $$ \tilde{F}(\tau, X)=  { \sqrt{T} \over K } F(\tau T, \sqrt{T}\, X).$$
           In particular, with the choice
           $$ K=  \sqrt{T}^{ 1- ( {2 \over p } + {3 \over q} ) } \|F \|_{L^p_{T}L^q},$$
            we get that $\|\tilde{F} \|_{L^p_{1} L^q}=1$  and thus that
             $$  \| f \|_{L^\infty_{T} L^\infty}  =  K\|\tilde{f}\|_{L^\infty_{1}L^\infty}
              \leq M  K = M  \sqrt{T}^{ 1- ( {2 \over p } + {3 \over q} ) } \|F \|_{L^p_{T}L^q}.$$
            This  gives the part of the estimate \eqref{Nash} involving $F$. 
            
      Let us turn to the study of the second problem in \eqref{td*} in order to get the part of the estimate
       \eqref{Nash} involving $G$. The estimate can be deduced from the previous one
         when $q_{1}<+ \infty$ which is the interesting case (when $q_{1}= \infty$, the estimate
          is a direct   consequence of the Maximum principle).
        Indeed, if $G \in {L_{T}^{p_{1}}}L^{q_{1}}$, we can write
         $G = \nabla \cdot F$ with $ F \in  {L_{T}^{p_{1}}} W^{1, q_{1} }  \subset L_{T}^{p_{1}} L^{q_{1}^*}$
          by Sobolev embedding (where  $q_{1}^*= { 3 q_{1} \over 3 - q_{1}}.$) and moreover, 
           we have
           $$ \| F \|_{L_{T}^{p_{1} }  L^{ q_{1}^* }} \leq  \| G \|_{L^{p_{1}}_{T}L^{q_{1}}}.$$
          Consequently, by  using the estimate that we have already proven, we get that
          $$ \|f \|_{L^\infty_{T}L^\infty} \leq  M   \sqrt{T}^{ 1- ( {2 \over p_{1} } + {3 \over q_{1}^* } )  }
           \|F \|_{L_{T}^{p_{1}}L^{q_{1}^*} }$$
            if  $2/p_{1}+ 3/ q_{1}^* <1$.  This  gives the claimed estimate.
            
        It remains to study the third problem in \eqref{td*} that  is the problem with no source term
         but a nontrivial initial data. Again, 
          we shall first  prove that there exists $M>0$ such that for  every  $f_{0} \in L^r$ with
           $ \|f_{0}\|_{L^2} \leq 1$, we have the estimate
           \begin{equation}
          \label{T=10}
           \sup_{t \geq 1 } \|f (t)  \|_{L^ \infty} \leq M. 
           \end{equation}
         the standard energy estimate gives that
        \begin{equation}
        \label{L2init}
         \| f \|_{L^\infty L^2}^2 +  \|f \|_{L^2L^2}^2 \leq \|f_{0} \|_{L^2}.
         \end{equation}
         To improve this estimate, we shall also use the De Giorgi-Nash iteration method. We  take a sequence
          $M_{k}$ as previously, and we also choose a sequence of times $T_{k}= { 1 - {1 \over k+1}}$
           which tends to $1$. The energy estimate for  $( f-M_{k})_{+}$ yields that
            for every $t, \, s$ with $t \geq T_{k} \geq s$, we have
            $$  \sup_{t \geq T_{k}} \| (f- M_{k})_{+}(t)  \|_{L^2}^2 + \int_{s}^{+ \infty} \| \nabla ( f - M_{k})_{+}) (\tau)
             \|_{L^2}^2 \, d \tau \leq \| (f-M_{k})_{+} (s) \|_{L^2}^2$$
              and hence, by integrating in $s$ for $T_{k-1} \leq s \leq T_{k}$, we obtain that
         \begin{equation}
         \label{Uk0}   U_{k} \leq { (k+ 1)^2} \int_{T_{k-1}}^{+ \infty}  \| (f-M_{k})_{+} (s) \|_{L^2}^2\, ds
         \end{equation}
           with 
           $$U_{k} =   \sup_{t \geq T_{k}} \| (f- M_{k})_{+}(t)  \|_{L^2}^2 + \int_{T_{k}}^{+ \infty} \| \nabla ( f - M_{k})_{+}) (\tau)
             \|_{L^2}^2 \, d \tau.$$
              The aim is again to estimate the right-hand side of \eqref{Uk0} by  a power of $U_{k-1}$
               strictly greater than $1$.  By using the same notations as previously, we
                 get from  \eqref{mk0} that 
             \begin{equation}
             \label{finit0}  U_{k} \leq  (k+ 1)^2  \Big({  (k+ 1 )^2 \over M }\Big)^{4\over 3 } \int_{T_{k-1}}^{+ \infty} \int_{\mathbb{R}^3} (f - M_{k-1})_{+}^{10 \over 3 }  \, dt dx \leq 
        (k+ 1)^2  \Big({  (k+ 1 )^2 \over M }\Big)^{4\over 3 } U_{k-1}^{4 \over 3}.\end{equation}
         Since  we have from \eqref{L2init}
  that $U_{0} \leq \big\| \big( f_{0}-(M - 1)\big)_{+}\big\|_{L^2}^2 \leq \|f_{0}\|_{L^2}^2$, $U_{1}$ can be made arbitrarily small
   by  taking $M$ sufficiently large and hence we get from  \eqref{finit0} that $\lim_{k \rightarrow + \infty}U_{k}= 0$.
     This proves that $f(t,x) \leq M$ for $t \geq 1$ and  then we get  \eqref{T=10}
      by changing $f$ into $-f$.
       We have thus proven that  for every $t \geq 1$ the linear operator $f \mapsto f(t, \cdot)$
        is bounded from $L^2$ into $L^\infty$ with norm smaller than $M$. Since by the standard
         maximum principle, it is also bounded from $L^\infty$ to $L^\infty$ with norm $1$, 
          we get  by interpolation that it also  maps $L^r$ to $L^\infty$ for every $r \geq 2.$  
         To get the  claimed estimate in \eqref{Nash} for $t \leq 1 $, it suffices to use again a scaling argument.

         This ends the proof.

 \end{proof}

\section{Proof of Lemma \ref{commu}}
\label{appendixB}
   \begin{proof}
Set $\phi=\mathcal{F}^{-1}h,$ then we have by definition and from Taylor formula
\begin{eqnarray*}
\big[h(\textnormal{D}),f \big]g&=&\int_{\RR^d} \phi(x-y)g(y)\big(f(y)-f(x)  \big)dy\\
&=&\int_0^1\int_{\RR^d} g(y)\Phi(x-y)\cdot\nabla f(x+t(y-x))dydt,
\end{eqnarray*}
with  $\Phi(x)=x\phi(x).$ Let $\alpha,\beta\in]0,1[$ with $\alpha+\beta=1.$ Using H\"older inequality and a change of variables  we get with $\Phi_t=t^{-3}\Phi(\frac{x}{t})$
\begin{eqnarray*}
\big|\big[h(\textnormal{D}),f \big]g(x)\big|&\le& \int_0^1\int_{\RR^d} \big(|g(y)||\Phi(x-y)|^\alpha\big) \big( |\nabla f(x+t(y-x))||\Phi(x-y)|^\beta \big)dydt\\&\le&\Big(\int_{\RR^3}|\Phi(x-y)||g(y)|^{\frac1\alpha}\big)dy\Big)^\alpha\int_0^1\int_{\RR^3}|\Phi_t(y)||\nabla f(x-y)|^{\frac1\beta}dy\Big)^\beta\\
&\le&\big(|\Phi|\star|g|^{\frac1\alpha}\big)^\alpha(x)\,\int_0^1\big(|\Phi_t|\star|\nabla f|^{\frac1\beta}\big)^{\beta}(x)dt.
\end{eqnarray*}
Let $p_1,p_2\in[1,\infty]$ such 
\begin{equation}\label{eq771}
\alpha p_1, \beta p_2\geq 1\quad\hbox{and}\quad\frac1p=\frac{1}{p_1}+\frac{1}{p_2}.
\end{equation} Then by H\"{o}lder inequality
$$
\big\|\big[h(\textnormal{D}),f \big]g\big\|_{L^p}\le\big\||\Phi|\star|g|^{\frac1\alpha}\big\|_{L^{\alpha p_1}}^\alpha\int_0^1\big\| |\Phi_t|\star|\nabla f|^{\frac1\beta} \big\|_{L^{\beta p_2}}^\beta dt.
$$
We choose $p_1,p_2,\alpha$ and $\beta$ such that
\begin{equation}\label{eq772}
\alpha m\geq1,\beta\rho\geq1,\,1+\frac{1}{\alpha p_1}=\frac1r+\frac{1}{\alpha m}\quad\hbox{and}\quad1+\frac{1}{\beta p_2}=\frac1r+\frac{1}{\beta\rho}
\end{equation}
then the classical convolution laws give
\begin{eqnarray*}
\big\|\big[h(\textnormal{D}),f \big]g\big\|_{L^p}&\le&\|\Phi\|_{L^r}^\alpha\|g\|_{L^m}\|\nabla f\|_{L^\rho}\int_0^1\|\Phi_t\|_{L^r}^\beta dt\\
&\le&\|\Phi\|_{L^r}\|g\|_{L^m}\|\nabla f\|_{L^\rho}\int_0^1t^{3\beta(-1+\frac{1}{r})}dt.
\end{eqnarray*}
This last integral is finite despite that 
\begin{equation}\label{ez1}
\beta<\frac13\frac{r}{r-1}\cdot
\end{equation} 
 Now let us check that the set given by conditions \eqref{eq771}, \eqref{eq772} and \eqref{ez1}  is not empty. First for the case $r=1$ we choose $p_1=m, p_2=\rho, \alpha=\frac1m$ and $\beta=1-\alpha.$ Let us now discuss the case $r>1.$ From \eqref{eq772}
$$
\alpha=\frac{r}{r-1}\big(\frac{1}{m}-\frac{1}{p_1}\big).
$$
To get $\alpha\in[\frac1m,1[$ we must choose $p_1$ such that
\begin{equation}\label{eq7774}
\frac1p-\frac1\rho<\frac{1}{p_1}\le\frac{1}{rm}.
\end{equation}
The condition $\beta\rho\geq1$ is equivalent, by the use of $1+\frac1p=\frac1m+\frac1\rho+\frac1r$, to
\begin{equation}\label{eq7775}
\frac1p-\frac{1}{r\rho}\le\frac{1}{p_1}.
\end{equation} 
The condition $\alpha p_1\geq1$ is automatically satisfied from \eqref{eq7774} since
$$
\alpha p_1\geq r\alpha m\geq  r\geq1.
$$
The condition $\beta p_2\geq1$ is also a consequence of \eqref{eq7774} and \eqref{eq7775}. Indeed,  from the value of $\alpha$ and $\frac1p=\frac{1}{p_1}+\frac{1}{p_2}$, this condition is equivalent to
$$
\frac1p-\frac{r}{2r-1}\frac1\rho\leq\frac{1}{p_1}\cdot
$$
This condition is weaker than \eqref{eq7775}.
We can easily check that \eqref{eq7774} and \eqref{eq7775} are equivalent to
$$
\frac1p-\frac{1}{r\rho}\leq\frac{1}{p_1}\le\frac{1}{rm}.
$$ The set of $p_1$ described by the above condition is nonempty if 
$$
\frac{1}{rm}-\frac1p+\frac{1}{r\rho}\geq0.
$$
Using the identity $1+\frac1p=\frac1m+\frac1\rho+\frac1r$, this is satisfied  under the condition $p\geq r$.  The condition \eqref{ez1}  is equivalent to
$$
\frac{1}{p_1}<\frac13+\frac1p-\frac1\rho.
$$ Now there is a compatibility between  this condition and \eqref{eq7775} if
$$
3(1-\frac1r)<\rho.
$$
This ends  the proof of Lemma \ref{commu}.
\end{proof}


\end{document}